\documentclass[11pt,reqno]{amsart}

\usepackage{amsmath,amssymb,amsfonts,bm,amscd,amsthm}
\usepackage{verbatim}
\usepackage[dvips]{graphicx}
\usepackage{epsfig,subfigure}
\usepackage{psfrag}
\usepackage{mathrsfs}
\usepackage{fancyhdr}
\usepackage{setspace}

\newtheorem{thm}{Theorem}
\newtheorem{lem}[thm]{Lemma}

\newtheorem{exmp}[thm]{Example}
\newtheorem{rem}[thm]{Remark}
\newtheorem{defn}[thm]{Definition}
\newtheorem{assum}[thm]{Assumption}
\newtheorem{prob}{Problem}

\newcommand{\scale}{1}

\newcommand{\tbroucke}{\hspace{-.01in}$\sim$\hspace{-.01in}broucke}
\newcommand{\Zero}{{\mathbf{0}}}

\newcommand{\RR}{{\mathbb R}}
\newcommand{\TT}{{\mathbb T}}

\newcommand{\cB}{{\mathcal B}}
\newcommand{\cC}{{\mathcal C}}

\newcommand{\cF}{{\mathcal F}}
\newcommand{\cG}{{\mathcal G}}
\newcommand{\cH}{{\mathcal H}}

\newcommand{\cO}{{\mathcal O}}

\newcommand{\cS}{{\mathcal S}}

\newcommand{\cY}{{\mathcal Y}}
\newcommand{\Bs}{{\mathscr{B}}}
\newcommand{\Ms}{{\mathscr{M}}}
\newcommand{\Zs}{{\mathscr{Z}}}

\newcommand{\mh}{\widehat{m}}

\newcommand{\ol}{\overline}

\renewcommand{\Im}{\textup{Im}}
\newcommand{\conv}{\textup{co}}

\newcommand{\cone}{\textup{cone}}
\newcommand{\spn}{\textup{sp}}
\newcommand{\rnk}{\textup{rank}}

\begin{document}

\title{Reach Control on Simplices by Piecewise Affine Feedback} 
\author{Mireille~E.~Broucke and Marcus~Ganness}
\address{Dept. of Electrical and Computer Engineering \\
University of Toronto, Toronto ON Canada M5S 3G4}
\email{broucke@control.utoronto.ca}
\date{\today}
\maketitle

\begin{comment}
\title{Reach Control on Simplices by Piecewise Affine Feedback} 
\author{Mireille~E.~Broucke and Marcus~Ganness% <-this % stops a space
\thanks{Manuscript revised \today.
        This work was supported by the Natural Sciences and
        Engineering Research Council of Canada (NSERC).}% <-this % stops a space
\thanks{Edward S.\ Rogers Sr.\ Department of
        Electrical and Computer Engineering, University of Toronto,
        Toronto, ON M5S 3G4, Canada (e-mail: broucke@control.utoronto.ca)}}
\maketitle
\end{comment}

\begin{abstract}                          % Abstract of not more than 250 words.
We study the reach control problem for affine systems on simplices, and the focus 
is on cases when it is known that the problem is not solvable by continuous state feedback.
We examine from a geometric viewpoint the structural properties of the system which 
make continuous state feedbacks fail. This structure is encoded by so-called
{\em reach control indices}, which are defined and developed in the paper. Based on
these indices, we propose a subdivision algorithm and associated piecewise affine feedback.
The method is shown to solve the reach control problem in all remaining cases, assuming
it is solvable by open-loop controls. 
\end{abstract}

%%%%%%%%%%%%%%%%%%%%%%%%%%%%%%%%%%%%%%%%%%%%%%%%%%%%%%%%%%%%%%%%%%%%%%%%%%%%%%%%%%%%%%%%%%%%%%%%%%%%%%%%%%%%%%%%%%%%%%

\section{Introduction}
\label{sec:intro}

This paper studies the {\em reach control problem} (RCP) on simplices. The problem is for trajectories of an
affine system defined on a simplex to reach a prespecified facet of the simplex in finite time. The overall concept
of the problem and its setting were introduced in \cite{HVS01} and further developed in \cite{HVS04, HVS06, RB06, MEB10}. 
The significance of the problem stems from its capturing the essential features of reachability problems 
for control systems: the presence of state constraints and the notion of trajectories reaching a goal 
in a guided and finite-time manner. The problem fits within a larger family of reachability problems; namely,
to reach a target set $\mathcal{X}_f$ with state constraint in a set $\mathcal{X}$, 
denoted as $\mathcal{X} \overset{\mathcal{X}}{\longrightarrow} \mathcal{X}_f$.
In the present context, we assume that the state constraints give rise to a state space that is 
triangulable \cite{LEE}; then the reachability specification is converted 
to a sequence of reachability problems on simplices of the triangulation. 
The reader is referred to \cite{MEB10,HVS01,HVS04,HVS06,RB06,LB11,BELTA} for further motivations, 
including how the studied problem arises in fundamental problems concerning hybrid systems \cite{TEEL}. 

RCP is one among several different research paths for analysis and synthesis of piecewise affine (PWA) feedback 
\cite{EXPLICIT,CAMLIBEL,SCHUTTER}. Recent progress on explicit MPC schemes has 
fueled the interest in PWA feedbacks \cite{EXPLICIT}, such feedbacks play a prominent role in linear switched systems 
\cite{LIBERZON}, and PWA systems have significant applications in engineering and biology \cite{TAHAMI,CASEY,OKTEM,BECK}.
A feature of our approach is that, rather than directly computing a controller numerically, we seek conditions for existence of controllers based on the problem data. 
This follows classical lines of thought which are well established in control theory. 
Another classical underpinning is to exploit system structure to understand 
the limits of a control system, again distinguishing our approach from numerical methods.

\section{Contributions}

In \cite{MEB10} it was shown that, under a suitable triangulation of the state space, 
affine feedback and continuous state feedback are 
equivalent from the point of view of solvability of the reach control problem (RCP). The approach
is based, fundamentally, on fixed point theory. The latter allows to deduce that continuous
state feedbacks always generate closed-loop equilibria inside the simplex when affine feedbacks do. 
The current paper departs from these findings, and using a geometric approach, 
we explore the system structure that gives rise to equilibria. This structure is encoded in
so-called {\em reach control indices}. The first goal of this paper is to elucidate these indices. 
The second goal is to use the indices
to obtain a subdivision of the simplex and an associated piecewise affine feedback to solve RCP 
in those cases when the problem is not solvable by continuous state feedback. It is shown that RCP is 
solvable by piecewise affine feedback if it is solvable by open-loop controls. 
This finding gives strong evidence to the relevance of the class of piecewise affine 
feedbacks in solving reachability problems.

\begin{figure}[!t]
\begin{center}
\psfrag{AA}[][][\scale]{(a)}
\psfrag{BB}[][][\scale]{(b)}
\psfrag{S}[][][\scale]{$\cS$}
\psfrag{S1}[][][\scale]{$\cS^1$}
\psfrag{S2}[][][\scale]{$\cS^2$}
\psfrag{xb}[][][\scale]{$\ol{x}$}
\psfrag{vp}[][][\scale]{$v'$}
\psfrag{v0}[][][\scale]{$v_0$}
\psfrag{v1}[][][\scale]{$v_1$}
\psfrag{v2}[][][\scale]{$v_2$}
\psfrag{y0}[][][\scale]{$y_0$}
\psfrag{b1}[][][\scale]{$b_1$}
\psfrag{b2}[][][\scale]{$b_2$}
\psfrag{F0}[][][\scale]{$\cF_0$}
\psfrag{F1}[][][\scale]{$\cF_1$}
\psfrag{F2}[][][\scale]{$\cF_2$}
\psfrag{O}[][][\scale]{$\cO$}
\psfrag{ConeS}[][][\scale]{$\cone(\cS)$}
\psfrag{ConeS1}[][][\scale]{$\cone(\cS^1)$}
\includegraphics[width= 0.9\linewidth]{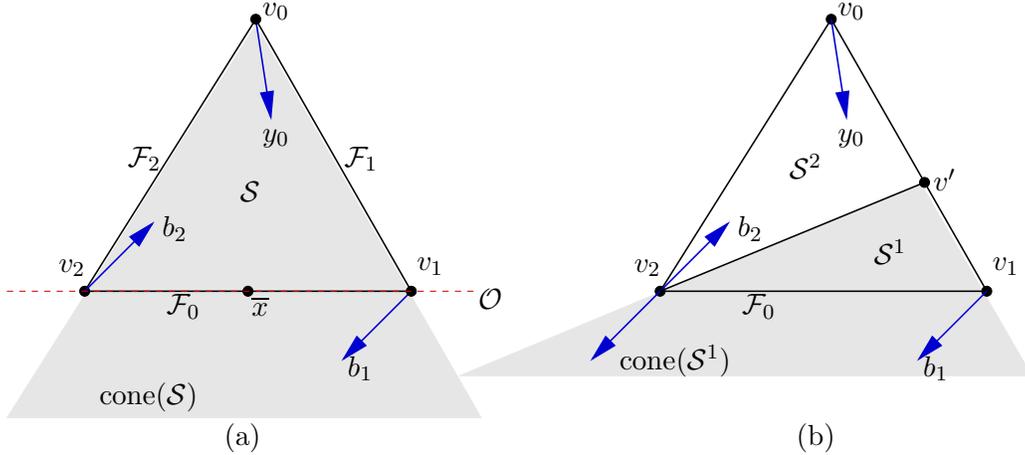}
\caption{Main idea in a 2D example.}
\label{fig0}
\end{center}
\end{figure}

The main ideas of the paper can be understood informally. Consider a 2D simplex $\cS = \conv\{ v_0,v_1,v_2 \}$ 
the convex hull of vertices $v_0$, $v_1$, and $v_2$, with 
1D facets $\cF_0$, $\cF_1$, and $\cF_2$, as in Figure~\ref{fig0}(a). 
Consider a single-input control system $\dot{x} = A x + b u + a$ defined on $\cS$. 
The {\em reach control problem} is to find a state feedback $u = f(x)$ such that all closed-loop trajectories 
initialized in $\cS$ leave $\cS$ in finite time through the {\em exit facet} $\cF_0$. 
The procedure to solve this control problem by continuous state feedback is to select control values $u_i$ 
at the vertices $v_i$ such that the velocity vectors $A v_i + b u_i + a$ point inside $\cone(\cS)$, the
cone with apex at $v_0$ determined by $\cS$; otherwise trajectories may leave $\cS$ 
through $\cF_1$ or $\cF_2$, which is disallowed. 
The controller $u =f(x)$ is formed as a continuous interpolation of the control values at the vertices. 
Label the vertex velocity vectors as $y_0 = A v_0 + b u_0 + a$, and
$b_i = A v_i + bu_i + a$, $i = 1,2$, as in the figure. 
Suppose that $A x + a \in \Im(b)$ along a line $\cO$ 
through $v_1$ and $v_2$. Clearly closed loop equilibria can only appear 
on the set $\cG := \cS \cap \cO$. Now it is obvious that this 
control problem cannot be solved by any continuous state feedback. 
For at $v_1$, $b_1$ has to point down to be inside $\cone(\cS)$, but at $v_2$, $b_2$ has to point up.
If we continuously interpolate along $\cF_0$ from $v_1$ to $v_2$, the continuous vector field, 
always in $\Im(b)$ along $\cF_0$, must pass through zero (by the Intermediate Value Theorem) at 
some $\ol{x}$ along $\cF_0$. The defect is that there are two vertices $v_1$ and $v_2$ that ``share'' 
the only control direction available, $b$. 

Suppose we now allow discontinuous feedback. 
Place a point $v'$ along the edge from $v_0$ to $v_1$ and define a new simplex $\cS^1= \conv \{ v', v_1, v_2 \}$.
See Figure~\ref{fig0}(b). Notice that as we slide $v'$ from $v_0$ to $v_1$ the cone $\cone(\cS^1)$ with apex at $v'$
widens at $v_2$ enough that $-b_2$ points inside $\cone(\cS^1)$ at $v_2$. Notice also that $v_1$ is unaffected by
sliding $v'$. Pick such a $v'$. Then one can construct 
an affine feedback $u = K^1 x + g^1$ on $\cS^1$ that assigns a non-zero velocity vector at every point on $\cF_0$, 
so there is no closed loop equilibrium in $\cS^1$. By \cite{HVS06,RB06}, RCP is solved on $\cS^1$. 
For the remaining simplex $\cS^2$ it is also possible to devise an affine controller so there is no equilibrium in
$\cS^2$. This is because equilibria can only appear in $\cS^2$ at $v_2 \in \cO$. But at $v_2$ we can select 
the velocity vector $b_2 \neq 0$. Again RCP can be solved on $\cS^2$ by affine feedback. Combining the two
affine feedbacks, we get a discontinuous piecewise affine feedback that solves RCP on $\cS$. Note that a discontinuity
is introduced because we use two different control values at $v_2$.

The contribution of the paper is to make mathematically rigorous the informal ideas described above. 
The main technical difficulty arises in dealing with multi-input systems. For this we bring in two tools. 
First we introduce the reach control indices to group together vertices in $\cG$ that share control inputs. These
indices are similar in spirit to the controllability indices to group together states 
that share control inputs \cite{BRUNOV}. As with the controllability indices, the reach control indices
require a special ordering of a set of linearly independent vectors; however, other technical details 
are different. 
The second tool is $\Ms$-matrices which help to concisely represent the constraints on the vector field 
at vertices of $\cG$. The reader is referred to Chapter 6 of \cite{BERMAN} for relevant background. 

The paper is organized as follows.
In Section~\ref{sec:back} we review the reach control problem.
In Section~\ref{sec:necessary} we give necessary conditions for solvability by open-loop controls.
These then shape the assumptions to construct the reach control indices, 
which are developed in Section~\ref{sec:indices}.
In Section~\ref{sec:pwa}, a subdivision method and associated piecewise
affine feedback are proposed to solve RCP when
continuous state feedback does not. The main result is presented in Section~\ref{sec:pwa} showing the relationship 
between solvability via open-loop controls and solvability via piecewise affine feeback. 
Examples are presented in Section~\ref{sec:examples}. 
Preliminary versions of parts of this paper appeared in \cite{NOLCOS10,ACC11}.
Proofs of supporting lemmas are found in the Appendix. 

{\em Notation.} \
For $x \in \RR^n$, the notation $x \succ 0$ ($x \succeq 0$) means $x_i > 0$ ($x_i \ge 0$) for $1 \le i \le n$.
The notation $x \prec 0$ ($x \preceq 0$) means $-x \succ 0$ ($-x \succeq 0$).
%For a matrix $A \in \RR^{n \times n}$, the notation $A \succ 0$ ($A \succeq 0$) means $a_{ij} > 0$ ($a_{ij} \ge 0$)
%for $1 \le i,j \le n$. 
Notation $\Zero$ denotes the subset of $\RR^n$ containing only the zero vector.
The notation $\Bs$ denotes the open unit ball, and $\ol{\Bs}$ denotes its closure.
The notation $\conv\{ v_1,v_2,\ldots \}$ denotes the convex hull of a set of points $v_i \in \RR^n$,
and $\spn\{ y_1,y_2,\ldots \}$ denotes the span of vectors $y_i \in \RR^n$.
The notation $(v_i,v_j)$ denotes the open segment in $\RR^n$ between $v_i, v_j \in \RR^n$. 
Finally, $T_{\cS}(x)$ denotes the Bouligand tangent cone to set $\cS$ at a point $x$ \cite{CLARKE}.

\section{Problem Statement}
\label{sec:back}

Consider an $n$-dimensional simplex $\cS$, the convex hull of $n+1$ affinely independent 
points in $\RR^n$. Let its vertex set be $V := \{ v_0, \dots, v_n \}$ and its 
facets $\mathcal{F}_0, \dots, \mathcal{F}_n$. The facet will be indexed by the vertex it 
does not contain. Let $h_j \in \RR^n$, $j=0, \dots, n$ be the unit normal vector to each facet
$\mathcal{F}_j$ pointing outside of the simplex. Facet 
$\mathcal{F}_0$ is called the {\em exit facet} of $\mathcal{S}$.
Define the index set $I := \{ 1, \ldots, n \}$. For $x \in \cS$ defined the closed, convex cone 
\[ 
\cC(x) := \{ y \in \RR^n ~|~ h_j \cdot y \leq 0, ~j \in I ~~s.t.~~ x \in \cF_j \} \,. 
\]
We'll write $\cone(\cS) := \cC(v_0)$ because $\cC(v_0)$ is the tangent cone to $\cS$ at 
$v_0$. We consider the affine control system on $\cS$:
\begin{equation}
\label{eq:thesystem}
\dot x =A x + Bu + a \,, 
\end{equation}
where $A \in \mathbb{R}^{n \times n}$, $a \in \mathbb{R}^n$, $B \in \mathbb{R}^{n \times m}$, 
and $\rnk (B)=m$. Let $\mathcal{B} = \Im (B)$, the image of $B$. 
Define $\cO := \{ ~ x \in \RR^n ~|~ A x + a \in \cB \}$ and $\cG := \cS \cap \cO$. 
Note that closed-loop equilibria of \eqref{eq:thesystem} can only appear in $\cO$.
Let $\phi_u(t,x_0)$ denote the trajectory of \eqref{eq:thesystem} starting at $x_0$ under
input $u$. 

\begin{exmp}
\label{example1}
Consider Figure~\ref{fig1} where we illustrate the notation in a 2D example. We have a full-dimensional 
simplex in $\RR^2$ given by $\cS = \conv \{ v_0,v_1,v_2 \}$ with vertex set $V = \{ v_0,v_1,v_2 \}$ 
and facets $\cF_0$,$\cF_1$, and $\cF_2$. Each facet $\cF_j$ has an outward normal 
vector $h_j$. The only vertex not in facet $\cF_j$ is vertex $v_j$. $\cF_0$ is the exit
facet. If we assume that $v_0 = 0$, then subspace $\cB$ is shown passing through $v_0$. 
The set $\cO$ is an affine space shown passing through $\cF_0$. Notice in this 
case $\cG = \cS \cap \cO = \conv\{ v_1, v_2 \}$. The cone $\cone(\cS)$ is the cone with apex at $v_0$ 
determined by $\cS$. It is indicated in the figure as the shaded area. 
The cones $\cC(v_i)$, $i = 0,1,2$ are depicted as darker shaded cones attached at each vertex. Of course, the
apex of each $\cC(v_i)$ is at the origin, but we depict it as being attached at the corresponding vertex $v_i$
since it will be used to describe allowable directions for the vector field at the vertices. Notice 
that the cones $\cC(v_1)$ and $\cC(v_2)$ are not tangent cones to $\cS$ at $v_1$ and $v_2$, respectively, whereas
$\cC(v_0)$ is the tangent cone to $\cS$ at $v_0$; hence the distinguished labeling of $\cC(v_0)$ 
as $\cone(\cS)$.
\end{exmp}

\begin{figure}[!t]
\begin{center}
\psfrag{v0}[][][\scale]{$v_0$}
\psfrag{v1}[][][\scale]{$v_1$}
\psfrag{v2}[][][\scale]{$v_2$}
\psfrag{h0}[][][\scale]{$h_0$}
\psfrag{h1}[][][\scale]{$h_1$}
\psfrag{h2}[][][\scale]{$h_2$}
\psfrag{S}[][][\scale]{$\cS$}
\psfrag{O}[][][\scale]{$\cO$}
\psfrag{B}[][][\scale]{$\cB$}
\psfrag{C0}[][][\scale]{$\cC(v_0)$}
\psfrag{C1}[][][\scale]{$\cC(v_1)$}
\psfrag{C2}[][][\scale]{$\cC(v_2)$}
\psfrag{F0}[][][\scale]{$\cF_0$}
\psfrag{F1}[][][\scale]{$\cF_1$}
\psfrag{F2}[][][\scale]{$\cF_2$}
\psfrag{ConeS}[][][\scale]{$\cone(\cS)$}
%\subfigure[]{\includegraphics[width = 0.4\linewidth]{fig1.eps}} \hspace{.2in}
%\subfigure[]{\includegraphics[width = 0.4\linewidth]{fig1.eps}}
\includegraphics[width = 0.5\linewidth]{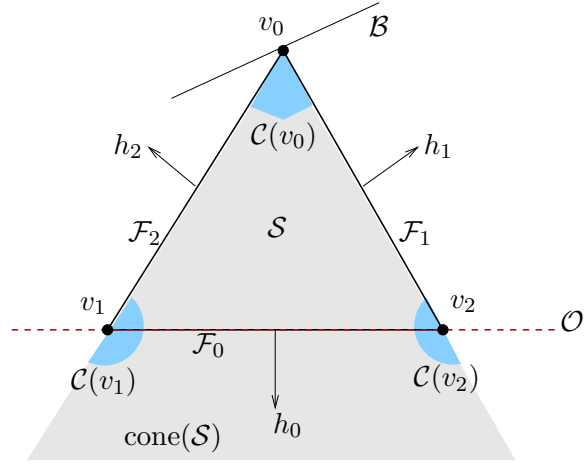}
\caption{Notation for reach control problem.}
\label{fig1}
\end{center}
\end{figure}

We are interested in formulating a problem to make the closed-loop trajectories of 
\eqref{eq:thesystem} exit $\cS$ through the exit facet $\cF_0$ only. 
For this, we require conditions that disallow trajectories to exit from
any other facet $\cF_i$, $i \in I$. 
We say the invariance conditions are solvable at vertex $v_i \in V$ if there exists
$u_i \in \RR^m$ such that 
\begin{equation}
\label{eq:inv}
Av_i + Bu_i + a \in \cC(v_i) \,. 
\end{equation}
We say the invariance conditions are solvable if \eqref{eq:inv} is solvable at each $v_i \in V$. 
The inequalities \eqref{eq:inv} are called {\em invariance conditions}. They guarantee trajectories
cannot exit from the facets $\cF_i$, $i \in I$, and they are used to construct 
affine feedbacks \cite{HVS04}. For general state feedbacks, stronger conditions (also called
invariance conditions) are needed. 
We say a state feedback $u = f(x)$ satisfies the invariance conditions if 
for all $x \in \cS$,
\begin{equation}
\label{eq:inv2}
Ax + Bf(x) + a \in \cC(x) \,. 
\end{equation}

\begin{exmp}
\label{example2}
Consider Figure~\ref{fig2}. Attached at each vertex is a velocity vector 
$y_i := A v_i + B u_i + a$, $i \in \{ 0 \} \cup I$. The invariance conditions \eqref{eq:inv} require that 
$y_i \in \cC(v_i)$, as illustrated. Notice that velocity vectors at $v_i \in \cF_0$ may or may not point
out of $\cS$. If the control is an affine feedback $u = K x + g$ such
that $u_i = K v_i + g$, then by convexity of the closed-loop vector field, \eqref{eq:inv2} 
holds at every $x \in \cF_i$, $i \in I$. If the input is a continuous state feedback $u = f(x)$, 
then invariance conditions for every $x \in \cF_i$, $i \in I$, must be explicitly stated, since convexity
is not guaranteed; hence \eqref{eq:inv2}. 
\end{exmp}

\begin{figure}[!t]
\begin{center}
\psfrag{v0}[][][\scale]{$v_0$} 
\psfrag{v1}[][][\scale]{$v_1$} 
\psfrag{v2}[][][\scale]{$v_2$} 
\psfrag{y0}[][][\scale]{$y_0$}
\psfrag{y1}[][][\scale]{$y_1$} 
\psfrag{y2}[][][\scale]{$y_2$} 
\psfrag{S}[][][\scale]{$\cS$} 
\psfrag{O}[][][\scale]{$\cO$}
\psfrag{C0}[][][\scale]{$\cC(v_0)$}
\psfrag{C1}[][][\scale]{$\cC(v_1)$}
\psfrag{C2}[][][\scale]{$\cC(v_2)$}
\psfrag{F0}[][][\scale]{$\cF_0$}
\psfrag{F1}[][][\scale]{$\cF_1$}
\psfrag{F2}[][][\scale]{$\cF_2$}
%\subfigure[]{\includegraphics[width = 0.4\linewidth]{fig1.eps}} \hspace{.2in}
%\subfigure[]{\includegraphics[width = 0.4\linewidth]{fig2.eps}}
\includegraphics[width = 0.4\linewidth]{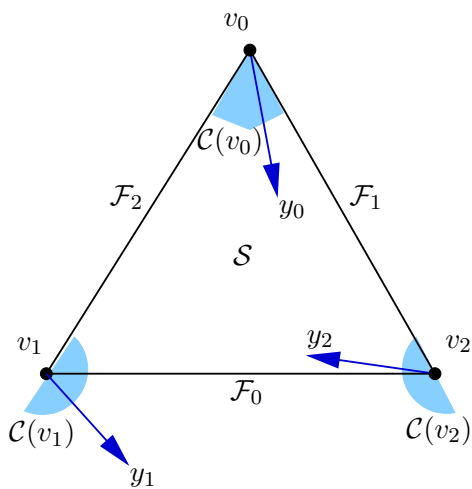}
\caption{The invariance conditions state that $y_i := A v_i + B u_i + a \in \cC(v_i)$ for $i =0,\ldots,n$.}
\label{fig2}
\end{center}
\end{figure}

\begin{prob}[Reach Control Problem (RCP)]
\label{prob0}
Consider system (\ref{eq:thesystem}) defined on $\mathcal{S}$.
Find a state feedback $u = f(x)$ such that:
\begin{enumerate}
\item[(i)] 
For every $x \in \cS$ there exist $T \geq 0 $ and $\gamma > 0$ such that
$\phi_u(t,x) \in \cS$ for all $t \in [0, T]$, $\phi_u(T,x) \in \cF_0$, and
$\phi_u(t,x) \notin \cS$ for all $t \in (T, T+\gamma)$.
\item[(ii)] 
There exists $\varepsilon > 0$ such that for every $x \in \cS$,
$~~\| A x + B f(x) + a \| > \varepsilon$. 
\item[(iii)]
Feedback $u = f(x)$ satisfies the invariance conditions \eqref{eq:inv2}. 
\end{enumerate}
\end{prob}
Condition (i) is the same condition that appears in the standard formulation of RCP \cite{HVS06,RB06}.
It states that all closed-loop trajectories must exit $\cS$ through $\cF_0$ in finite time without
first exiting from another facet. Condition (ii) and (iii) are new, 
and they are introduced to deal with pathologies that can only happen
when using discontinuous feedbacks. It can be shown that if continuous state feedback 
is used, then condition (i) implies conditions (ii) and (iii) \cite{HVS04}.
Therefore, results on affine feedbacks \cite{HVS06,RB06} and continuous state feedbacks 
\cite{MEB10} remain valid.

\begin{exmp}
\label{rem1}
In this example we illustrate the need for condition (ii). Figure~\ref{fig3}(a) illustrates
a 2D simplex $\cS = \conv \{ v_0, v_1, v_2 \}$, where $v_0 = (1,1)$, $v_1 = (0,0)$, and $v_2 = (2,0)$.
We consider the linear system on $\cS$:
\[
\dot{x} = 
\left[
\begin{array}{rr}
2 & -2 \\
1 & -2 
\end{array}
\right] x + 
\left[
\begin{array}{r} 2 \\ 1 \end{array} 
\right] u \,.
\]
Now we select control values $u_0 = 0$, $u_1 = 0$, and $u_2 = -3$ to satisfy the invariance conditions of $\cS$,
and we solve for the feedback $u = -\frac{3}{2}x_1 + \frac{3}{2} x_2$. The closed loop system 
has an equilibrium in $\cS$ at $v_1$ only, and if we compute the time to reach $\cF_0$ from any $x_0 \in \cS$ 
we find it is finite, despite the presence of the equilibrium. 
Now we define a discontinuous piecewise affine feedback $u = f(x)$ given by 
$f(x) := -\frac{3}{2}x_1 + \frac{3}{2} x_2$ for  $x \in \cS \setminus \{ v_1 \}$ and $f(v_1) := -1$. 
Not only do all trajectories reach $\cF_0$ in finite time, they also exit $\cS$ as required by condition (i).
We have a feasible solution to RCP, but it is not structurally stable. If the 
system parameters $( A,B,a )$ are slightly perturbed and we use $u = f(x)$, then there can 
appear an equilibrium $\ol{x}$ of the perturbed system in the interior of $\cS$, 
as shown in Figure~\ref{fig3}(b). Thus, RCP is not solved for the perturbed system. 
Condition (ii) disallows such non-robust solutions. 
\end{exmp}

\begin{exmp}
\label{rem2}
Next consider Figure~\ref{fig3}(c) which represents a second pathological solution to RCP using
discontinuous feedback. Here trajectories reach $\cF_0$ in finite time, and then they slide along $\cF_0$ out of the
simplex along a direction at $v_2$ that violates $v_2$'s invariance conditions. In order
to circumvent this behavior, it is sufficient to disallow feedbacks that violate the invariance
conditions \eqref{eq:inv2}, particularly on $\cF_0$. This is the purpose of condition (iii). 
\end{exmp}

In the sequel we will use the shorthand notation $\cS \overset{\cS}{\longrightarrow} \cF_0$ 
to denote that (i)-(iii) of Problem~\ref{prob0} hold under some control law. 
Finally, we make an important assumption concerning the placement of $\cO$ with respect to $\cS$. 
The reader is referred to \cite{MEB10} for the motivation and a method of triangulation 
of the state space that achieves this assumption. See also \cite{LEE}.
\begin{assum} 
\label{assum1}
Simplex $\cS$ and system \eqref{eq:thesystem} satisfy the following condition:
if $\cG \neq \emptyset$, then $\cG$ is a $\kappa$-dimensional face of $\cS$, 
where $0 \le \kappa \le n$.
\end{assum}

\begin{figure}[!t]
\begin{center}
\psfrag{v0}[][][\scale]{$v_0$} 
\psfrag{v1}[][][\scale]{$v_1$} 
\psfrag{v2}[][][\scale]{$v_2$} 
\psfrag{y1}[][][\scale]{$y_1$} 
\psfrag{S}[][][\scale]{$\cS$} 
\psfrag{A}[][][\scale]{(a)} 
\psfrag{B}[][][\scale]{(b)} 
\psfrag{C}[][][\scale]{(c)} 
\psfrag{xb}[][][\scale]{$\ol{x}$} 
\psfrag{F0}[][][\scale]{$\cF_0$}
\psfrag{F1}[][][\scale]{$\cF_1$}
\psfrag{F2}[][][\scale]{$\cF_2$}
%\subfigure[]{\includegraphics[width = 0.4\linewidth]{fig1.eps}} \hspace{.2in}
%\subfigure[]{\includegraphics[width = 0.4\linewidth]{fig1.eps}}
\includegraphics[width = 0.8\linewidth]{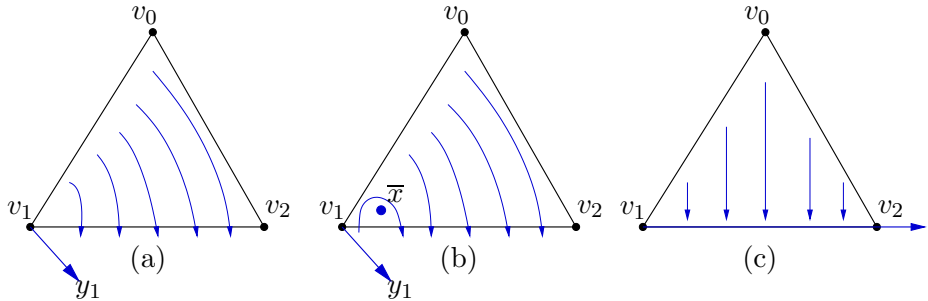}
\caption{Illustration of pathologies that arise using discontinuous feedback to solve RCP.}
\label{fig3}
\end{center}
\end{figure}

\section{Necessary Conditions}
\label{sec:necessary}

In this section we present two necessary conditions for solvability of RCP using open-loop controls.
We take as open-loop controls for \eqref{eq:thesystem} any measurable function 
$\mu: [0,\infty) \rightarrow \RR^m$ that is bounded on compact intervals. 
Now we define what is meant by a solution of RCP by open-loop controls.

\begin{defn}
\label{def:ol}
Consider system (\ref{eq:thesystem}) defined on $\mathcal{S}$.
We say {\em $\cS \overset{\cS}{\longrightarrow} \cF_0$ by open-loop controls}
if there exists a map $T : \cS \rightarrow \RR^+$ and a set of open-loop controls 
$\{ \mu_x ~|~ x \in \cS \}$ such that:
\begin{enumerate}
\item[(i)] 
For every $x \in \mathcal{S}$ there exists $\gamma > 0$ such that
$\phi_{\mu_{x}}(t,x) \in \mathcal{S}$ for all $t \in [0, T(x)]$, 
$\phi_{\mu_{x}}(T(x),x) \in \mathcal{F}_0$, and
$\phi_{\mu_{x}}(t,x) \notin \mathcal{S}$ for all $t \in (T(x), T(x)+\gamma)$.
\item[(ii)] 
There exists $\varepsilon > 0$ such that for every $x \in \cS$ and $t \in [0, T(x)]$,
$\| A \phi_{\mu_{x}}(t,x) + B \mu_{x}(t)  + a \| > \varepsilon$. 
\item[(iii)]
For every $x \in \cS$ and $t \in [0, T(x)]$, 
$(A \phi_{\mu_x}(t,x) + B \mu_x(t) + a) \in \cC({\left(\phi_{\mu_x}(t,x) \right)})$. 
\end{enumerate}
\end{defn}

The first result of the section is that solvability of the invariance conditions \eqref{eq:inv} is necessary
for solvability of RCP by open-loop controls in the sense of condition (i) only. This extends the 
analogous result in \cite{HVS04} on the necessity of the invariance conditions for solvability of RCP (in the
sense of condition (i) only) for continuous state feedbacks. Proofs are in the Appendix.

\begin{thm}
\label{thm:nec1}
If $\cS \overset{\cS}{\longrightarrow} \cF_0$ by open-loop controls in the sense of condition (i) only, 
then the invariance conditions \eqref{eq:inv} are solvable. 
\end{thm}

The second result says that if RCP is solvable by open-loop controls, then it is possible to assign
non-zero velocity vectors satisfying \eqref{eq:inv} at vertices $v_i \in V \cap \cG$. This is an immediate 
consequence of condition (ii). 
We know that $A v_i + a \in \cB$ for vertices $v_i \in \cG$. Theorem~\ref{thm:nec1} says that if
RCP is solvable by open-loop controls (in the sense of condition (i)), then $\cB \cap \cC(v_i) \neq \emptyset$, 
for $v_i \in V \cap \cG$. The next result says that, moreover, the zero vector cannot be the only element of 
$\cB \cap \cC(v_i)$, $v_i \in V \cap \cG$. 

\begin{thm}
\label{thm:nec2}
If $\cS \overset{\cS}{\longrightarrow} \cF_0$ by open-loop controls, then 
$\cB \cap \cC(v_i) \neq \Zero$, $v_i \in V \cap \cG$. 
\end{thm}

\section{Reach Control Indices}
\label{sec:indices}

The reach control indices are defined in the situation when it is known that RCP
is not solvable by continuous state feedback but it is still solvable by open-loop controls. 
According to the results of \cite{MEB10}, RCP is not solvable by continuous state feedback
under the following assumptions.
\begin{assum} 
\label{assum2}
Simplex $\cS$ and system \eqref{eq:thesystem} satisfy the following conditions.
\begin{enumerate}
\item[(A1)] 
$\cG = \cS \cap \cO = \conv\{ v_1,\ldots,v_{\kappa+1} \}$, with $0 \le \kappa < n$.
\item[(A2)] 
$\cB \cap \cone(\cS) = \Zero$.
\item[(A3)] 
The maximum number of linearly independent vectors in any set 
$\{ b_1,\ldots,b_{\kappa+1} ~|~ b_i \in \cB \cap \cC(v_i) \}$
(with only one vector for each $\cB \cap \cC(v_i)$) is $\mh$ with $0 \le \mh < \kappa + 1$. 
\item[(A4)]
$\cB \cap \cC(v_i) \neq \Zero$, $i = 1,\ldots,\kappa+1$. 
\end{enumerate}
\end{assum}

Assumption (A1) restricts $\cG$ to be a face in $\cF_0$. The other cases arising under Assumption~\ref{assum1} 
when $\cG = \emptyset$ or $v_0 \in \cG$ are trivially solvable (Theorems 6.1 and Remark 7.1 of \cite{MEB10}). 
If either (A2) or (A3) does not hold, then RCP is solved by affine feedback (Theorems 6.2 and 6.7 of \cite{MEB10}). 
Assumption (A4) is no loss of generality due to Theorem~\ref{thm:nec2}. 

\begin{exmp}
Consider Figure~\ref{fig1}. We have $\cG = \conv \{ v_1, v_2 \}$, which
satisfies (A1). Notice (A1) is a strengthening of 
Assumption~\ref{assum1} - it imposes that $v_0 \not\in \cG$; otherwise RCP is not solvable \cite{MEB10}. 
%Referring to the figure, suppose we had $v_0 \in \cG$. Then to achieve \eqref{eq:inv} at $v_0$, one must
%find $u_0$ so that $A v_0 + B u_0 + a \in \cB \cap \cone(\cS) = \Zero$ (by (A2)). 
%Thus, the only solution is to place an equilibrium at $v_0$. 
(A2) is also illustrated in Figure~\ref{fig1}. At $v_0$, $\cB$ has no vectors
in common with $\cone(\cS)$ except the zero vector. 
Next, we see that (A3) is satisfied with $\mh = m = 1$. In particular, 
$b_1 \in \cB \cap \cC(v_1)$ and $b_2 \in \cB \cap \cC(v_2)$ are linearly dependent. 
Note also that (A3) specifies that $\mh < \kappa + 1$. 
If $\mh = \kappa + 1$, then RCP is solvable by affine feedback \cite{MEB10}. 
Finally (A4) is taken from Theorem~\ref{thm:nec2}. It says that at each vertex in $\cG$, there exists 
a non-zero $b_i \in \cB$ satisfying the invariance conditions of $v_i$ for $i = 1,2$. 
\end{exmp}

Consider (A3). Select any $b_i \in \cB \cap \cC(v_i)$, $i = 1,\ldots,\kappa+1$ and write the
list $\{ b_1,\ldots,b_{\kappa+1} \}$. Clearly there exists a list with a maximum number $\mh$ 
of linearly independent vectors. W.l.o.g., we reorder the indices $\{ 1,\ldots,\kappa+1\}$ 
(leaving the indices $0,\kappa+2,\ldots,n$ the same) 
so that $\{ b_1,\ldots,b_{\mh} \}$ are linearly independent. 
By (A3), $\mh < \kappa + 1$ so we can define $p \ge 1$ as 
\[
p := \kappa + 1 - \mh \,.
\]
Notice by the maximality of $\{ b_1,\ldots,b_{\mh} \}$, for each $i = \mh+1,\ldots,\kappa+1$ and for each
$b_i \in \cB \cap \cC(v_i)$, $b_i \in \spn \{ b_1,\ldots,b_{\mh} \}$.
Now consider the cone $\cB \cap \cC(v_{\mh+1})$. 
By (A4), $\cB \cap \cC(v_{\mh+1}) \neq \Zero$, 
so there exists $2 \le r_1 \le \mh+1$ such that w.l.o.g. (reordering indices $1,\ldots,\mh$),
$\cB \cap \cC(v_{\mh+1}) \subset \spn \{ b_1,\ldots,b_{r_1-1} \}$
and $\spn \{ b_1,\ldots,b_{r_1-1} \}$ is the smallest subspace generated by basis vectors among
$\{ b_1,\ldots,b_{\mh} \}$ only that contains the cone $\cB \cap \cC(v_{\mh+1})$. 

\begin{lem}
There exists a unique minimal subspace containing $\cB \cap \cC(v_{\mh+1})$ and generated by the basis
$\{ b_1,\ldots,b_{\mh} \}$.
\end{lem}

In order to have consecutive indices, it is useful to renumber the vertices of $\cG$ to effectively 
swap the indices $\mh+1$ and $r_1$, so we get 
\begin{equation}
\label{eq:coner1a}
\cB \cap \cC(v_{r_1}) \subset \spn \{ b_1, \ldots, b_{r_1-1} \} \,.
\end{equation}
The following establishes that one can always find a vector in 
$\cB \cap \cC(v_{r_1})$ that depends on all the vectors in $\{ b_1,\ldots,b_{r_1-1} \}$.

\begin{lem}[\cite{NOLCOS10}]
\label{lem4}
Suppose Assumption~\ref{assum2} and \eqref{eq:coner1a} hold. 
There exists $\ol{b}_{r_1} \in \cB \cap \cC(v_{r_1})$ such that
\begin{equation}
\label{eq:coner1b}
\ol{b}_{r_1} = c_1 b_1 + \cdots + c_{r_1-1} b_{r_1-1} \,, \qquad c_i \neq 0, ~~i = 1,\ldots,r_1-1 \,.
\end{equation}
\end{lem}

We now have a list
\begin{equation}
\label{eq:rci2}
\{ b_1,\ldots,b_{r_1-1},\ol{b}_{r_1},b_{r_1+1},\ldots,b_{\mh+1} \} \,.
\end{equation}
The overbar on $\ol{b}_{r_1}$ reminds us that it depends on all the previous $r_1-1$ vectors 
in the list. For this reason, with $b_{r_1} := \ol{b}_{r_1}$, 
$\{ b_2, \ldots, b_{\mh+1} \}$ are linearly independent. 
The next result is a direct implication of condition (A2).

\begin{lem}
\label{lem:cneg}
Suppose Assumption~\ref{assum2} and \eqref{eq:coner1a}-\eqref{eq:coner1b} hold. 
Then the coefficients in \eqref{eq:coner1b} satisfy $c_i < 0$, $i = 1,\ldots,r_1-1$.
\end{lem}

\begin{proof}
Suppose w.l.o.g. (by reordering indices $\{ 1,\ldots,r_1-1\}$), there exists $1 \le \rho < r_1-1$
such that $c_i > 0$ for $i = 1,\ldots,\rho$ and $c_i <0$ for $i = \rho+1,\ldots,r_1-1$. 
Consider the vector
$\beta := \ol{b}_{r_1} - c_{\rho+1} b_{\rho+1} - \cdots - c_{r_1-1} b_{r_1-1} 
      = c_1 b_1 + \cdots + c_{\rho} b_{\rho}$. 
Notice that $\beta \neq 0$ since $\{ b_1,\ldots,b_{\rho} \}$ are linearly independent. 
Since $b_i \in \cB \cap \cC(v_i)$, $i \in \{1,\ldots,r_1\}$, we have
$h_j \cdot \beta = h_j \cdot \bigl( \ol{b}_{r_1} - c_{\rho+1} b_{\rho+1} - \cdots - c_{r_1-1} b_{r_1-1} \bigr) \le 0$, 
$j = 1,\ldots,\rho,r_1+1,\ldots,n$. 
Also $h_j \cdot \beta = h_j \cdot \bigl( c_1 b_1 + \cdots + c_{\rho} b_{\rho} \bigr) \le 0$, 
$j = \rho+1,\ldots,n$. 
In sum, $h_j \cdot \beta \le 0$, $i \in I$; that is, $\beta \in \cB \cap \cone(\cS)$. By Assumption (A2),
$\beta = 0$, a contradiction.
\end{proof}

\begin{rem}
An notable feature of Lemma~\ref{lem:cneg} is that any $b_i$, $i = 1,\ldots,r_1$, can be expressed as a negative
linear combination of the remaining vectors $\{ b_1,\ldots,b_{i-1},b_{i+1},\ldots,b_{r_1} \}$. This means we
may renumber indices within $\{ 1,\ldots, r_1 \}$ with impunity, as the formula \eqref{eq:coner1b} will still 
hold with strictly negative coefficients. Such a renumbering will be invoked in Lemma~\ref{lem10} of the next 
section.
\end{rem}

At this point we have $r_1$ cones $\cB \cap \cC(v_i)$, $i = 1,\ldots,r_1$, and we have a selection 
$\{ b_1,\ldots,b_{r_1} ~|~ b_i \in \cB \cap \cC(v_i) \}$ with the property that any $r_1-1$ vectors in
the selection is linearly independent, and each vector of the selection is a strictly negative linear combination
of the others. This situation creates strong restrictions on $\cB$. Indeed for the vectors 
$\{ b_1,\ldots,b_{r_1} \}$ to meet these properties and to lie in their respective cones, they 
have a special geometric relationship with $\cS$, as described next. 

\begin{lem}
\label{lem1}
Suppose Assumption~\ref{assum2} and \eqref{eq:coner1a}-\eqref{eq:coner1b} hold. Then 
\begin{align}
\label{eq:bi} 
h_j \cdot b_i = 0 \,, \qquad &i = 1, \ldots, r_1 \,, \qquad 
                              j \in I \setminus \{ 1, \ldots, r_1 \}  \,.
\end{align}
\end{lem}

\begin{proof}
Let $b_{r_1}$ be as in \eqref{eq:coner1b}. Since $b_{r_1} \in \cB \cap \cC(v_{r_1})$, 
$h_j \cdot b_{r_1} = h_j \cdot \bigl( c_1 b_1 + \cdots + c_{r_1-1} b_{r_1-1} \bigr) \le 0$, 
$j \in I \setminus \{ 1,\ldots,r_1 \}$. 
Since $b_i \in \cB \cap \cC(v_i)$ and, by Lemma~\ref{lem:cneg}, $c_i < 0$, 
every term in the sum is non-negative. The result immediately follows. 
\end{proof}

\begin{figure}[!t]
\begin{center}
\psfrag{v0}[][][\scale]{$v_0$}
\psfrag{v1}[][][\scale]{$v_1$}
\psfrag{v2}[][][\scale]{$v_2$}
\psfrag{v3}[][][\scale]{$v_3$}
\psfrag{b1}[][][\scale]{$b_1$}
\psfrag{b2}[][][\scale]{$b_2$}
\psfrag{h1}[][][\scale]{$h_1$}
\psfrag{h3}[][][\scale]{$h_3$}
\psfrag{B}[][][\scale]{$\cB$}
\psfrag{S}[][][\scale]{$\cS$}
\psfrag{G}[][][\scale]{$\cG$}
\psfrag{o}[][][\scale]{$\cO$}
\psfrag{cone}[][][\scale]{$\cone(\cS)$}
\includegraphics[width= 0.5 \linewidth]{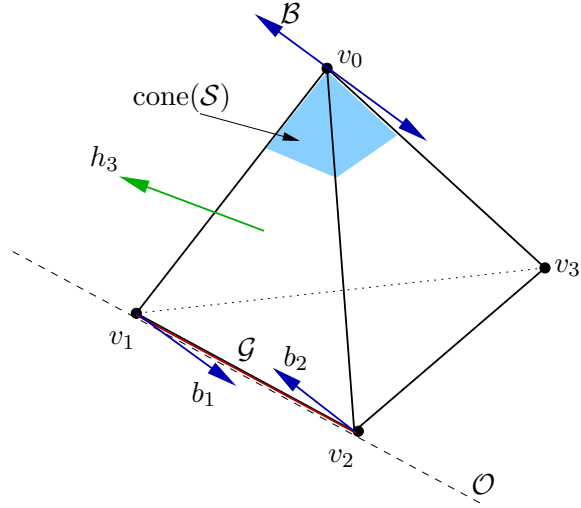}
\caption{Illustration for Lemma~\ref{lem1}.\label{figCSF2}}
\end{center}
\end{figure}

\begin{exmp}
Lemmas~\ref{lem:cneg} and \ref{lem1} are illustrated for a  3D example in Figure~\ref{figCSF2}. We have
$\cS = \conv \{ v_0,\ldots,v_3 \}$, $\cG = \cS \cap \cO = \conv \{ v_1, v_2 \}$, and with $v_0 = 0$
we see that $\cB \cap \cone(\cS) = \Zero$. Also, $\mh = m = 1$. Vector $b_i$ shown attached at $v_i$
lies in the cone $\cB \cap \cC(v_i)$, $i = 1,2$. Now we observe that $b_2 = -c_1 b_1$, $c_1 > 0$, to satisfy 
$b_i \in \cB \cap \cC(v_i)$. This is the content of Lemma~\ref{lem:cneg}. Second, we observe from the figure
that the only way $b_i \in  \cB \cap \cC(v_i)$, $i = 1,2$, can hold simultaneously is if 
$h_3 \cdot b_i = 0$, $i =1,2$. That is, $b_1$ and $b_2$ lie in the 2D plane containing
$\cF_3$. This is the content of Lemma~\ref{lem1}. 
\end{exmp}

Next we consider the cone $\cB \cap \cC(v_{\mh+2})$. Proceeding as above, there exists a smallest subspace 
generated by the basis $\{ b_2,\ldots,b_{\mh+1} \}$ that contains $\cB \cap \cC(v_{\mh+2})$. 
By independently reordering each index set $\{ 2,\ldots,r_1 \}$ and $\{ r_1+1,\ldots,\mh+1 \}$ we have
$\cB \cap \cC(v_{\mh+2}) \subset \spn \{ b_{\rho},\ldots,b_{\rho+r_2-2} \}$,
for some $2 \le \rho \le r_1+1$ and $\rho \le \rho + r_2 - 2 \le \mh+1$. 
Lemmas~\ref{lem4} and \ref{lem:cneg} can be adapted for $\cB \cap \cC(v_{\mh+2})$ 
since we have exactly the same situation as for $\cB \cap \cC(v_{r_1})$, only the indices are changed.
Thus, we get 
\begin{align}
\label{eq:coner2a}
\cB \cap \cC(v_{\mh+2}) & \subset \spn \{ b_{\rho}, \ldots, b_{\rho+r_2-2} \} \\
\label{eq:coner2b}
(\exists \ol{b}_{\mh+2} \in \cB \cap \cC(v_{\mh+2})) & \quad 
\ol{b}_{\mh+2} = c_{\rho} b_{\rho} + \cdots + c_{\rho+r_2-2} b_{\rho+r_2-2} \,, ~~c_i < 0 \,.
\end{align}
We can similarly invoke Lemma~\ref{lem1} to obtain
\begin{align}
\label{eq:bi2} 
h_j \cdot b_i = 0 \,, \quad &i = \rho, \ldots, \rho+r_2-2, \mh+2 \,, \quad 
                              j \in I \setminus \{ \rho, \ldots, \rho+r_2-2, \mh+2 \}  \,.
\end{align}

At this point we know $\rho \le r_1+1$. Next we show that actually $\rho = r_1 + 1$.
This means that the lists $\{ b_1,\ldots,b_{r_1} \}$ and $\{ b_{\rho},\ldots,b_{\rho+r_2-2} \}$
have no vectors in common.  The ensuing proof is facilitated by $\Ms$-matrices \cite{BERMAN}. 
Let $1 \le \alpha \le \beta \le \kappa+1$, $b_i \in \cB \cap \cC(v_i)$, 
and define $H_{\alpha,\beta} := [h_{\alpha} \cdots h_{\beta}]$, 
$Y_{\alpha,\beta} := [b_{\alpha} \cdots b_{\beta}]$, and 
$M_{\alpha,\beta} := H^T_{\alpha,\beta} Y_{\alpha,\beta}$. 
A matrix $M$ is a $\Zs$-matrix if the off-diagonal elements are non-positive; 
i.e. $m_{ij} \le 0$ for all $i \neq j$. A matrix $M$ is {\em monotone} if
$M c \preceq 0$ implies $c \preceq 0$. 
A $\Zs$-matrix $M$ is a nonsingular $\Ms$-matrix if it is monotone. 

\begin{lem}
\label{lem2}
Suppose Assumption~\ref{assum2} and \eqref{eq:coner1a}-\eqref{eq:coner1b} hold. 
Also suppose $\rho < r_1 + 1$. Then $M_{\rho,r_1}$ is a nonsingular $\Ms$-matrix.
\end{lem}

\begin{proof}
First, we know $M_{\rho,r_1}$ is a $\Zs$-matrix because $h_j \cdot b_i \le 0$, $j \neq i$, so 
the off-diagonal entries are non-positive. Second, we show $M_{\rho,r_1}$ is monotone.
let $c = (c_\rho,\ldots,c_{r_1})$ be such that $M_{\rho,r_1} c \preceq 0$.
Define $y := Y_{\rho,r_1} c$. Then $h_j \cdot y \le 0$, $j = \rho, \ldots, r_1$. 
Also by Lemma~\ref{lem1} and \eqref{eq:bi2}, $h_j \cdot y = 0$, $j = 1,\ldots,\rho-1,r_1+1, \ldots, n$. 
Thus, $y \in \cB \cap \cone(\cS)$. By (A2), $y = 0$. However, $\{ b_\rho,\ldots,b_{r_1} \}$ 
are linearly independent, so $c = 0$. Thus, $M_{\rho,r_1}$ is monotone. Finally,
by Theorem~6.2.3 case ($N_{39}$) of \cite{BERMAN}, $M_{\rho,r_1}$ is a nonsingular $\Ms$-matrix.
\end{proof}

\begin{rem}
A similar result to Lemma~\ref{lem2} first appeared in \cite{MEB10}; a step of the proof was clarified in
\cite{AB12}. Here we present a simpler argument based on monotonicity. 
\end{rem}

\begin{lem}
\label{lem12}
Suppose Assumption~\ref{assum2} and \eqref{eq:coner1a}-\eqref{eq:coner1b} hold. 
Then $\rho = r_1+1$. 
\end{lem}

\begin{proof}
Suppose by way of contradiction that $\rho < r_1+1$. Applying Lemma~\ref{lem1} we obtain 
\begin{subequations}
\begin{eqnarray}
h_j \cdot b_i & = & 0, ~~i = 1, \ldots, r_1 \,, \quad 
                         j = r_1+1, \ldots, n  
\label{eq:lem12a} \\
h_j \cdot b_i & = & 0, ~~i = \rho, \ldots, \rho + r_2 - 2,\mh+2 \,, 
\nonumber \\
              &   &    \quad j = 1, \ldots,\rho-1,\rho + r_2 - 1,\ldots, \mh+1, \mh+3, \ldots, n
\label{eq:lem12b} \,.
\end{eqnarray}
\end{subequations}
Let $\sigma = \min \{ r_1, \rho+r_2-2 \}$. 
Consider $M_{\rho,\sigma} = H_{\rho,\sigma}^T Y_{\rho,\sigma}$, 
where $Y_{\rho,\sigma} = [ b_{\rho} \cdots b_{\sigma} ]$. 
By Lemma~\ref{lem2}, $M_{\rho,r_1}$ is a nonsingular $\Ms$-matrix. 
By Theorem~6.2.3 and the remarks thereafter \cite{BERMAN}, $M_{\rho,\sigma}$ is also a nonsingular 
$\Ms$-matrix. By Theorem~6.2.3 (case $I_{28}$) of \cite{BERMAN} 
there exists $c' = (c_{\rho}',\ldots,c_{\sigma}')$ such that $c' \preceq 0$
and $M_{\rho,\sigma} c' \prec 0$. Define $\beta := Y_{\rho,\sigma} c' \neq 0$. The statement
$H_{\rho,\sigma}^T \beta = M_{\rho,\sigma} c' \prec 0$ is equivalent to
\begin{equation}
\label{eq:beta1}
h_j \cdot \beta < 0, \qquad j = \rho,\ldots,\sigma.
\end{equation} 
By \eqref{eq:lem12a}-\eqref{eq:lem12b},
\begin{equation}
\label{eq:beta2}
h_j \cdot \beta = h_j \cdot \left( c_{\rho}' b_{\rho} + \cdots + c_{\sigma}' b_{\sigma} \right) = 0 \,,
\qquad j = 1,\ldots,\rho-1,r_1+1,\ldots,n \,.
\end{equation}
If $\sigma < r_1$ then from \eqref{eq:lem12b}
\begin{equation}
\label{eq:beta3}
h_j \cdot \beta = 0, \qquad j = \sigma,\ldots,r_1 \,.
\end{equation} 
In sum, \eqref{eq:beta1}-\eqref{eq:beta3} imply $\beta \in \cB \cap \cone(\cS)$. 
By Assumption (A2), $\beta = 0$, a contradiction.
\end{proof}

Lemma~\ref{lem12} tells us that $\rho + r_2 - 1 = r_1 + r_2$ so variable $\rho$ will now be dropped. 
Now we renumber vertices of $\cG$ in order to effectively swap the indices $r_1 + r_2 $ and $\mh+2$, 
but we must take care that this index swap does not disturb
the foregoing construction for $\{ b_1,\ldots,b_{r_1} \}$. In particular, the two index sets $\{ 1,\ldots,r_1 \}$
and $\{ r_1+1,\ldots,n \}$ in \eqref{eq:bi} should not become entangled with each other.  
This is the case because $r_1 + r_2 > r_1 + 1$, so both $r_1 + r_2$ and $\mh+2$
belong to the index set $I \setminus \{ 1,\ldots,r_1 \}$. 

Two indices $r_1$ and $r_2$ have been put in place.
By iterating on Lemmas~\ref{lem1}, \ref{lem12}, and our index swap, we 
can further decompose $\cB$ relative to the cones $\cB \cap \cC(v_i)$ associated
with $\cG$. The procedure generates a specially ordered list of the form
\begin{equation}
\label{eq:rci3}
\{ b_1,\ldots,b_{r_1-1},\ol{b}_{r_1},b_{r_1+1},\ldots,b_{r_1+r_2-1},\ol{b}_{r_1+r_2},\ldots,
   b_{r_1+\cdots+r_{p-1}+1}, \ldots, b_{r-1}, \ol{b}_{r},b_{r+1},\ldots,b_{\kappa+1} \} \,, 
\end{equation}
where $b_i \in \cB \cap \cC(v_i)$ and $r := r_1 + \cdots + r_p$. 
The vectors that do not have an overbar are provided by (A3) (modulo the change of indices). 
The vectors that have overbars are provided by Lemma~\ref{lem4}, 
and each $\ol{b}_{r_1+\cdots+r_k}$ depends on all of the previous $r_k-1$
vectors in the list. 

\begin{thm}
\label{thm:indices}
Suppose Assumption~\ref{assum2} holds.  There exist integers 
$r_1,\ldots,r_p \ge 2$ such that w.l.o.g. (by reordering indices) 
\begin{subequations}
\begin{alignat}{2}
%\begin{array}{llll}
\cB \cap \cC(v_i) & ~\subset~ \spn \{ b_{m_1}, \ldots, b_{m_1+r_1-1} \} \,, 
                   & i &= m_1, \ldots, m_1+r_1-1 \,, \label{eq:r1} \\
                   & ~\vdots & ~~ &~~\vdots \nonumber \\
\cB \cap \cC(v_i) & ~\subset~ \spn \{ b_{m_p}, \ldots, b_{m_p+r_p-1} \} \,, \qquad 
                   & i &= m_p, \ldots, m_p+r_p-1 \,, \label{eq:rp} 
%\end{array} 
\end{alignat}
\end{subequations}
where $b_i \in \cB \cap \cC(v_i)$, $m_1 := 1$, and 
\begin{equation}
\label{eq:mk}
m_k := r_1 + \cdots + r_{k-1} + 1 \,, \qquad k = 2,\ldots,p \,.
\end{equation}
Moreover, for each $k = 1,\ldots,p$, $\{ b_{m_k}, \ldots, b_{m_k+r_k-2} \}$ are linearly independent
and
\begin{equation}
\label{eq:BG1}
b_{m_k+r_k-1} = c_{m_k} b_{m_k} + \cdots + c_{m_k+r_k-2} b_{m_k+r_k-2} \,, \quad c_i < 0 \,, 
i = m_k, \ldots, m_k+r_k-2 \,.
\end{equation}
\end{thm}

\begin{proof}
The vectors $b_i \in \cB \cap \cC(v_i)$ in \eqref{eq:r1}-\eqref{eq:rp} are provided by \eqref{eq:rci3},
including those from Lemma~\ref{lem4} (the overbar has now been removed). Lemma~\ref{lem:cneg} 
gives \eqref{eq:BG1}. It remains only to prove \eqref{eq:r1}-\eqref{eq:rp}. We consider only
\eqref{eq:r1}. Consider any $i \in \{ 1,\ldots,r_1 \}$ and any $\beta_i \in \cB \cap \cC(v_i)$ such that
$\beta_i = \alpha_1 b_1 + \cdots + \alpha_{r_1} b_{r_1} + \beta$, where $\alpha_i \in \RR$ 
and $\beta \in \cB$. W.l.o.g. we may assume $\beta$ is independent of 
$\{ b_1,\ldots,b_{r_1} \}$. From \eqref{eq:inv} and Lemma~\ref{lem1}, 
$h_j \cdot \beta_i 
= h_j \cdot (c_1 b_1 + \cdots + c_{r_1} b_{r_1} + \beta) 
= h_j \cdot \beta \le 0$, for $j = r_1+1,\ldots,n$. 
By the proof of Proposition~7.2 in \cite{MEB10}, $\beta = 0$. Hence, for any 
$i \in \{ 1,\ldots,r_1 \}$ and $\beta_i \in \cB \cap \cC(v_i)$, 
$\beta_i \in \spn \{ b_1, \ldots, b_{r_1} \}$, as desired. 
\end{proof}

The integers $\{ r_1,\ldots,r_p \}$ are called the {\em reach control indices} of 
system~\eqref{eq:thesystem} with respect to simplex $\cS$.

\begin{rem}
A number of relationships between the integers $\kappa$, $\mh$, $p$, and $r$ are implied by the construction.
By definition $p = \kappa+1 - \mh$. By (A3), $\mh \le \kappa$. Then we observe
that each of the ``excess'' $p$ vertices of $\cG$, namely $v_{m_1+r_1-1},...,v_{m_p+r_p-1}$, 
has an associated non-zero vector by (A4). By Lemma~\ref{lem12}, each of these $p$ vertices 
uses up at least one exclusive vector in the basis 
\[
\{ b_1,\ldots,b_{r_1-1},b_{r_1+1},\ldots,b_{r_1+r_2-1},\ldots, b_{r_1+\cdots+r_{p-1}+1}, \ldots, b_{r-1} \} \,.
\] 
So we need at least $p$ independent vectors in this basis. That is, 
\[
\mh \ge p =  \kappa+1 - \mh \,. 
\]
Thus, in order for (A4) to hold it is necessary that
\[
\mh \ge \frac{\kappa+1}{2} \,.
\]
This condition is interpreted to say that RCP is only solvable if there are sufficient inputs. 

\begin{comment}
Second, since each of the $p$ lists $\{ b_{m_i} , \ldots, b_{m_i+r_i-1} \}$, $i \in \{ 1,\ldots,p \}$,
comprises $r_i-1$ independent vectors, and there are a total of $\mh$ such vectors, we deduce that
\[
(r_1-1) + \cdots + (r_p-1) - p \le \mh \qquad \iff \qquad r - p \le \mh \,.
\]
\end{comment}

The construction does not make any explicit statements about the ``extra'' linearly independent vectors
$\{ b_{r+1},\ldots,b_{\kappa+1} \}$. These vectors correspond to cones that were ``swapped out'' 
due to the index renumbering. Moreover, the cones $\cB \cap \cC(v_i)$, $i = r+1,\ldots,\kappa+1$
do not enjoy the properties discovered for $\cB \cap \cC(v_i)$, $i = 1,\ldots,r$. 
\end{rem}

\section{Piecewise Affine Feedback}
\label{sec:pwa}

The reach control indices catalog the degeneracies (caused by insufficient inputs) 
that lead to the appearance of equilibria in $\cS$ whenever $p \ge 1$ and continuous state feedback is applied. 
Thus, any control method that overcomes the limits of continuous state feedback 
must confront this degeneracy and will necessarily draw
upon the degrees of freedom in $\cB$ provided to $\cG$ which are inscribed by the indices. 
In this section we investigate the extent to which piecewise affine feedback can solve RCP, in cases when 
continuous state feedback cannot. We construct a triangulation \cite{LEE} of the simplex $\cS$ 
such that RCP is solvable for each simplex of the triangulation. 
The next result shows that because of condition (iii) of Problem~\ref{prob0} a situation like the one in 
Figure~\ref{fig3}(c) cannot happen. Correspondingly one recovers a third necessary condition for solvability of
RCP by open-loop controls - in essence saying that $\cB$ cannot be parallel to $\cF_0$.

\begin{lem}
\label{lem:nec3}
Suppose Assumption~\ref{assum2} and \eqref{eq:r1}-\eqref{eq:BG1} hold.
If $\cS \overset{\cS}{\longrightarrow} \cF_0$ by open-loop controls, then 
$\spn \{ b_{m_k}, \ldots, b_{m_k+r_k-2} \} \not\subset \cH_0 := \{ y \in \RR^n ~|~ h_0 \cdot y = 0 \}$ 
for each $k = 1,\ldots,p$. 
\end{lem}

\begin{proof}
W.l.o.g. we consider only $k = 1$. Define
$\widehat{\cF}_0 := \conv \{ v_1,\ldots,v_{r_1} \} \subset \cF_0$. 
Let $\{ \mu_x \}$ be open-loop controls satisfying (i)-(iii) of Definition~\ref{def:ol}. 
Let $x \in \cS$ and consider any $t \in [0,T(x)]$ such that $\phi_{\mu_x}(t,x) \in \widehat{\cF}_0$. 
First, by condition (iii) of Definition~\ref{def:ol}, 
$h_l \cdot ( A \phi_{\mu_x}(t,x) + B \mu_x(t) + a ) \le 0$ for $l \in I$, $\phi_{\mu_x}(t,x) \in \cF_l$. 
Second, let $A \phi_{\mu_x}(t,x) + B \mu_x(t) + a = \alpha_1 b_1 + \cdots + \alpha_{r_1} b_{r_1} + \beta$, 
where $\alpha_i \in \RR$ and $\beta \in \cB$. By the same argument as in Theorem~\ref{thm:indices}, 
$\beta = 0$. Then by Lemma~\ref{lem1}, $h_j \cdot ( A \phi_{\mu_x}(t,x) + B \mu_x(t) + a ) = 0$ for
$j = r_1+1,\ldots, n$. Suppose by way of contradiction that $\spn \{ b_1, \ldots, b_{r_1} \} \subset \cH_0$. 
Then $h_0 \cdot ( A \phi_{\mu_x}(t,x) + B \mu_x(t) + a ) = 0$. 
On the other hand, for $z \in \widehat{\cF}_0$, 
$T_{\widehat{\cF}_0}(z) = \left\{ y \in \RR^n ~|~ 
                          h_j \cdot y = 0, h_l \cdot y \le 0, j = 0, r_1+1,\ldots,n, l \in I ~s.t.~
                          z \in \cF_l \right\}$. 
We conclude that for all $x \in \cS$ and $t \in [0,T(x)]$, if $\phi_{\mu_x}(t,x) \in \widehat{\cF}_0$,
then $A \phi_{\mu_x}(t,x) + B \mu_x(t) + a \in T_{\widehat{\cF}_0}(x)$. Using uniqueness of solutions,
we obtain $\widehat{\cF}_0$ is a positively invariant set, a contradiction.
\end{proof}

\begin{defn} 
\label{def:pwa}
Given system \eqref{eq:thesystem} and a state feedback $u = f(x)$, 
we say $f(x)$ is a {\em piecewise affine feedback} if there exists a 
triangulation $\TT$ of $\cS$ such that for each $n$-dimensional $\cS^j \in \TT$, there exist 
$K^j \in \RR^{m \times n}$ and $g^j \in \RR^m$ such that $f(x) = K^j x + g^j$, $x \in \cS^j$. 
\end{defn}

This definition of piecewise affine feedback allows for discontinuities at the boundaries of simplices;
moreover, the feedback is a multi-valued function, distinct from the usual notion of piecewise 
affine function in algebraic topology \cite{MUNKRES}. Resolving what control value to use at points lying in
more than one simplex is treated as a problem of implementation. The artifact of a 
{\em discrete supervisory controller} \cite{DES} will be introduced to convert 
the multi-valued function to a single-valued feedback. 

We now explain informally an inductive procedure for subdividing $\cS$ in order
that RCP can be solved by piecewise affine feedback. First, in Lemma~\ref{lem10} we show that because of
Lemma~\ref{lem:nec3}, each simplex $\conv \{ v_{m_k},\ldots,v_{m_k+r_k-1} \}$, $k = 1,\ldots,p$, has 
a vertex (among $\{ v_{m_k},\ldots,v_{m_k+r_k-1} \}$) with $b_i \in \cB \cap \cC(v_i)$ pointing
out of $\cS$. By convention, we reorder indices so this vertex is the first one in each list
$\{ v_{m_k},\ldots,v_{m_k+r_k-1} \}$. We make a subdivision of $\cS$ by placing a new vertex
$v'$ along the edge $(v_0,v_{m_k})$. In particular, at the first iteration we would 
have $v' \in (v_0,v_1)$, and we form two simplices $\cS^1$ and $\cS'$ as in Figure~\ref{fig4}. 
Lemma~\ref{lem9} shows that because $b_{m_k} \in \cB \cap \cC(v_{m_k})$ points out of $\cS$ at $v_{m_k}$ and  
because the invariance conditions for $\cS$ are solvable at $v_0$, a convexity argument 
(precisely, \eqref{eq:lem9}) gives that $v'$ can be placed along $(v_0,v_{m_k})$ so that 
$\cB \cap \cone(\cS^1) \neq \Zero$. Then in
Lemma~\ref{lem3} one applies Theorem~6.2 of \cite{MEB10} to obtain that RCP is solved for $\cS^1$. Essentially
$\cS^1$ can be removed from further consideration, and the induction step is repeated with $\cS$ replaced
by the remainder $\cS'$. See Figure~\ref{fig5}. To guarantee that the induction is sound, one must show that 
$\cS'$ inherits the relevant properties of $\cS$, especially the property of Lemma~\ref{lem:nec3}.
This is done in Lemma~\ref{lem5}.

\begin{lem}
\label{lem10}
Suppose Assumption~\ref{assum2} and \eqref{eq:r1}-\eqref{eq:BG1} hold.
Then w.l.o.g. (by reordering indices $\{ m_k,\ldots,m_k+r_k-1 \}$),
$h_0 \cdot b_{m_k} > 0$, $k = 1 , \ldots, p$. 
\end{lem}

\begin{proof} 
We prove the result only for $k = 1$. 
If for any $i \in \{ 1,\ldots,r_1 \}$, $h_0 \cdot b_i > 0$, then the proof is finished. Instead
suppose that for all $i \in \{ 1,\ldots,r_1 \}$, $h_0 \cdot b_i \le 0$. 
Using Lemma~\ref{lem:nec3} and by reordering the indices $1,\ldots,r_1$, assume
$h_0 \cdot b_{r_1} < 0$. By \eqref{eq:BG1}, 
$b_{1} = \frac{1}{c_1} \left( b_{r_1} - c_2 b_2 - \cdots - c_{r_1-1} b_{r_1-1} \right)$
with $c_i < 0$. Thus we obtain $h_0 \cdot b_1 = h_0 \cdot \frac{1}{c_1} 
\left( b_{r_1} - c_2 b_2 - \cdots - c_{r_1-1} b_{r_1-1} \right) \ge \frac{1}{c_1} h_0 \cdot b_{r_1} > 0$. 
\end{proof} 

\begin{exmp}
Lemmas~\ref{lem:nec3} and \ref{lem10} are illustrated in Figure~\ref{fig4} for a 2D example. We have
$\cG = \conv\{ v_1, v_2 \}$, $\cB \cap \cone(\cS) = \Zero$, and $\cB \cap \cC(v_i) \neq \Zero$, $i = 1,2$,
as required by Assumption~\ref{assum2}. We observe that $\cB$ is not parallel to $\cF_0$.
Otherwise, the only way for trajectories to exit $\cF_0$ would be by violating the invariance conditions at
$v_1$ or $v_2$ as depicted in Figure~\ref{fig3}(c). Therefore, $\cB$ cannot be parallel to $\cF_0$. 
This is the essence of Lemma~\ref{lem:nec3}. Next, since $\cB$ is not parallel to $\cF_0$ 
there is $b_1 \in \cB \cap \cC(v_1)$ that points out of $\cS$. This is the content of Lemma~\ref{lem10}.   
\end{exmp}

\begin{figure}[!t]
\begin{center}
\psfrag{v0}[][][\scale]{$v_0$} 
\psfrag{v1}[][][\scale]{$v_1$} 
\psfrag{v2}[][][\scale]{$v_2$} 
\psfrag{v1p}[][][\scale]{$v'$} 
\psfrag{b1}[][][\scale]{$b_1$} 
\psfrag{h0}[][][\scale]{$h_0$}
\psfrag{h1}[][][\scale]{$h_1$} 
\psfrag{h2}[][][\scale]{$h_2$} 
\psfrag{hp}[][][\scale]{$h'$} 
\psfrag{S1}[][][\scale]{$\cS'$} 
\psfrag{S2}[][][\scale]{$\cS^1$} 
\psfrag{O}[][][\scale]{$\cO$}
\psfrag{B}[][][\scale]{$\cB$}
\includegraphics[width= 0.5 \linewidth]{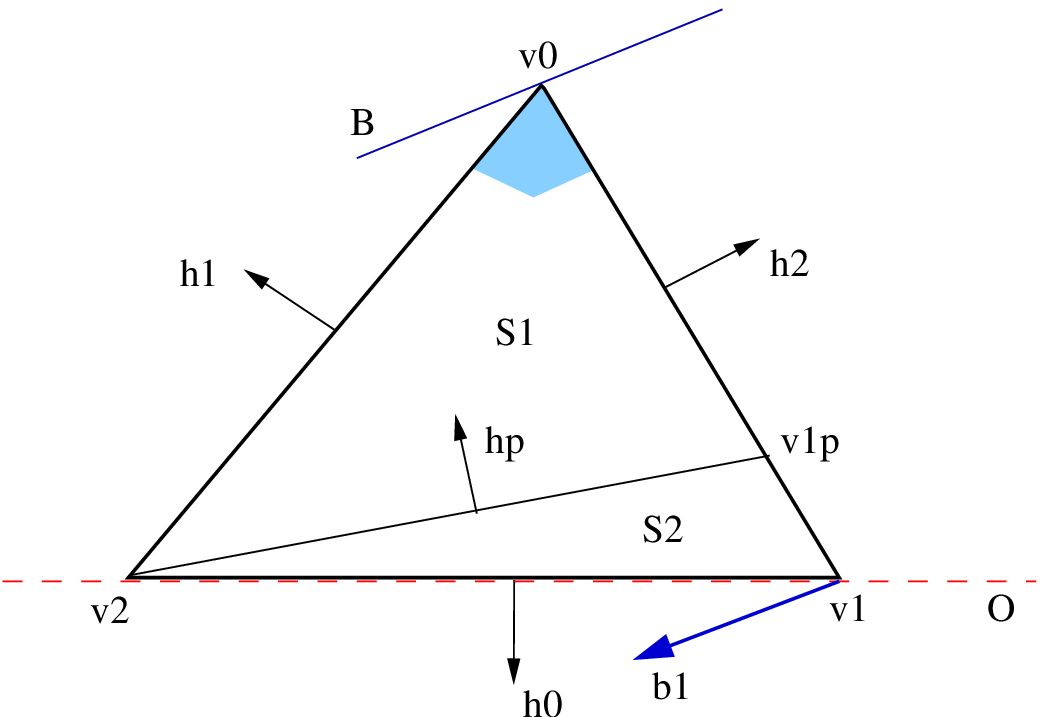}
\caption{Subdivision into two simplices $\cS'$ and $\cS^1$.}
\label{fig4}
\end{center}
\end{figure}

Following Lemma~\ref{lem10}, suppose that $b_1$ satisfies $h_0 \cdot b_1 > 0$. We consider any point 
$v'$ in the open segment $(v_0,v_1)$. That is, let $\lambda \in (0,1)$ and define
\begin{equation}
\label{eq:vprime}
v' = \lambda v_1 + (1-\lambda) v_0 \,.
\end{equation} 
Now define the following simplices in $\cS$:
\begin{eqnarray*}
\cS' & = & \conv\{ v_0, v', v_2,\ldots,v_n \} \\
\cS^1 & = & \conv\{ v', v_1, v_2,\ldots,v_n \} \,.
\end{eqnarray*}
Also define the new exit facet for $\cS'$ by
$\cF_0' := \conv\{ v', v_2,\ldots,v_n \}$. 
See Figure~\ref{fig4}. The following lemma provides a formula for the normal vector $h'$ of $\cF_0'$. 
\begin{lem}
\label{lem8}
Let $h_0 = -\gamma_1 h_1 - \ldots - \gamma_n h_n$ with $\gamma_i > 0$, and let $\lambda \in (0,1)$. 
Then the normal vector to $\cF_0'$ pointing out of $\cS^1$ is 
\begin{equation}
\label{eq:hprime}
 h' =  \gamma_1 h_1 + \lambda \sum_{j=2}^{n} \gamma_j h_j = \gamma_1 (1-\lambda) h_1 - \lambda h_0 \,.
\end{equation}
\end{lem}

\begin{lem}
\label{lem9}
Suppose Assumption~\ref{assum2} and \eqref{eq:r1}-\eqref{eq:BG1} hold.
There exists $v' \in (v_0,v_1)$, such that $\cB \cap \cone(\cS^1) \neq \Zero$. 
Moreover, $b_1 \in \cB \cap \cone(\cS^1)$ with $h' \cdot b_1 < 0$. 
\end{lem}

\begin{proof}
Observe that 
$
\cone(\cS^1) = \{ y \in \RR^n ~|~ h' \cdot y \leq 0 \,, h_j \cdot y \leq 0, \quad j \in \{ 2, \cdots, n\}\}. 
$
We show there is an interval of values for $\lambda$ such that $0 \neq b_1 \in \cB \cap \cone(\cS^1)$,
where we assume the index ordering of Lemma~\ref{lem10} so that $h_0 \cdot b_1 > 0$. 
First, since $b_1 \in \cB \cap \cC(v_1)$ we know $h_j \cdot b_1 \leq 0$ for $j \in \{ 2, \ldots, n\}$. 
We must only show that there exists $\lambda \in (0,1)$ such that $h' \cdot b_1 < 0$.
From Lemma~\ref{lem8} we have
\begin{equation}
\label{eq:lem9}
h' \cdot b_1 = \gamma_1 (1-\lambda) h_1 \cdot b_1 - \lambda h_0 \cdot b_1 \,.
\end{equation}
Since $h_1 \cdot b_1 > 0$ (because $\cB \cap \cone(\cS) = \Zero$) and $h_0 \cdot b_1 > 0$ (by Lemma~\ref{lem10}),
it is clear from \eqref{eq:lem9} that we can select $\lambda = \lambda'$ sufficiently close to $1$ such that 
$h' \cdot b_1 < 0$. Setting $v'=\lambda' v_1 + (1-\lambda')v_0$, we get $b_1 \in  \cB \cap \cone(\cS^1)$.
\end{proof}

\begin{lem} 
\label{lem3}
Suppose Assumption~\ref{assum2} and \eqref{eq:r1}-\eqref{eq:BG1} hold.
Let $v'$ be as in Lemma~\ref{lem9}. 
If the invariance conditions for $\cS$ are solvable,
then $\cS^1 \overset{\cS^1}{\longrightarrow} \cF_0$ by affine feedback.
\end{lem}

\begin{proof}
By Lemma~\ref{lem9}, we have $\cB \cap \cone(\cS^1) \neq \Zero$. We show that the
invariance conditions are solvable for $\cS^1$. First, consider the vertex $v'$. 
Since the invariance conditions for $\cS$ are solvable, there exist control inputs $u_0, u_1 \in \RR^m$ such 
that the invariance conditions for $\cS$ at $v_0$ and $v_1$ are satisfied, i.e.
$y_0 := A v_0 + B u_0 + a \in \cone(\cS)$ and 
$y_1 := A v_1 + B u_1 + a \in \cB \cap \cC(v_1)$. 
In particular, $h_j \cdot y_i \le 0$ for $i = 0, 1$ and $j = 2,\ldots,n$. Now by Lemma~\ref{lem9}, there exists 
$\lambda \in (0,1)$ such that with $v' := \lambda v_1 + (1 - \lambda) v_0$, $h' \cdot b_1 < 0$ and
$h_j \cdot b_1 \le 0$ for $j = 2,\ldots, n$. Let $w_1$ be such that $b_1 = B w_1$. Set $\epsilon_1 > 0$ and let
$u' := \lambda u_1 + (1-\lambda) u_0 + \epsilon_1 w_1$. 
Then $y' := A v' + B u' + a = \lambda y_1 + (1-\lambda) y_0 + \epsilon_1 b_1$.
Thus, $h_j \cdot y' \le 0$ for $j = 2,\ldots,n$ and for $\epsilon_1 > 0$ sufficiently large, $h' \cdot y' < 0$.
That is, the invariance conditions for $\cS^1$ are solvable at $v'$.

Next consider $v_1$. Since the invariance conditions for $\cS^1$ at $v_1$ are identical to those for $\cS$ at $v_1$, and since
the latter are by assumption solvable, the former are also solvable. 
Finally, consider vertices $v_i$, $i = 2,\ldots,n$. There exist control inputs $u_i \in \RR^m$ such that
$y_i := A v_i + B u_i + a$
satisfy $h_j \cdot y_i \le 0$ for $j = 2,\ldots,i-1,i+1,\ldots,n$. As above let $w_1$ be such that $b_1 = B w_1$. 
Set $\epsilon_1 > 0$ and let $u_i' := u_i + \epsilon_1 w_1$. 
Then $y_i' = A v_i + B u_i' + a = y_i + \epsilon_1 b_1$. 
Thus, $h_j \cdot y_i' \le 0$ for $j = 2,\ldots,i-1,i+1,\ldots,n$ and for $\epsilon_1 > 0$ sufficiently large, 
$h' \cdot y_i' < 0$. That is, the invariance conditions for $\cS^1$ are solvable at $v_i$.
In sum, we can apply Theorem~6.2 of \cite{MEB10} to obtain that 
$\cS^1 \overset{\cS^1}{\longrightarrow} \cF_0$ by affine feedback. 
\end{proof}

\begin{lem}
\label{lem5}
Suppose Assumption~\ref{assum2} and \eqref{eq:r1}-\eqref{eq:BG1} hold.
Let $v'$ be as in Lemma~\ref{lem9}. 
If the invariance conditions for $\cS$ are solvable then 
\begin{itemize}
\item[(i)] The invariance conditions for $\cS'$ are solvable.
\item[(ii)] $(-h') \cdot b_{m_k} > 0 \,, \quad k = 1,\ldots,p$.
\end{itemize}
\end{lem}

\begin{proof}
First we prove (i). By assumption the invariance conditions for $\cS$ are solvable, and since 
the invariance conditions for $\cS'$ are identical (the only facet that changed for $\cS'$ is $\cF_0$, 
which plays no role in invariance conditions), they are also solvable for $\cS'$. 
Next we prove (ii). First we have $(-h') \cdot b_{m_1} > 0$ by Lemma~\ref{lem9}. Second, 
since $b_{m_k} \in \cB \cap \cC(v_{m_k})$, we have $h_1 \cdot b_{m_k} \le 0$, for $k = 2, \ldots,p$. 
Also by Lemma~\ref{lem10}, $h_0 \cdot b_{m_k} > 0$, for 
$k = 2,\ldots,p$. Thus using \eqref{eq:hprime},
$(-h') \cdot b_{m_k} = -\gamma_1 (1 - \lambda) h_1 \cdot b_{m_k} + \lambda h_0 \cdot b_{m_k} > 0$, 
$k = 2, \ldots, p$. 
\end{proof}

\begin{figure}[!t]
\begin{center}
\psfrag{v0}[][][\scale]{$v_0$} 
\psfrag{v1}[][][\scale]{$v_1$} 
\psfrag{v2}[][][\scale]{$v_2$} 
\psfrag{v1p}[][][\scale]{$v'$} 
\psfrag{F0}[][][\scale]{$\cF_0$}
\psfrag{F01}[][][\scale]{$\cF_0^1$} 
\psfrag{h2}[][][\scale]{$h_2$} 
\psfrag{hp}[][][\scale]{$h'$} 
\psfrag{S}[][][\scale]{$\cS$} 
\psfrag{S1}[][][\scale]{$\cS^1$} 
\psfrag{O}[][][\scale]{$\cO$}
\psfrag{B}[][][\scale]{$\cB$}
\includegraphics[width= 0.9 \linewidth]{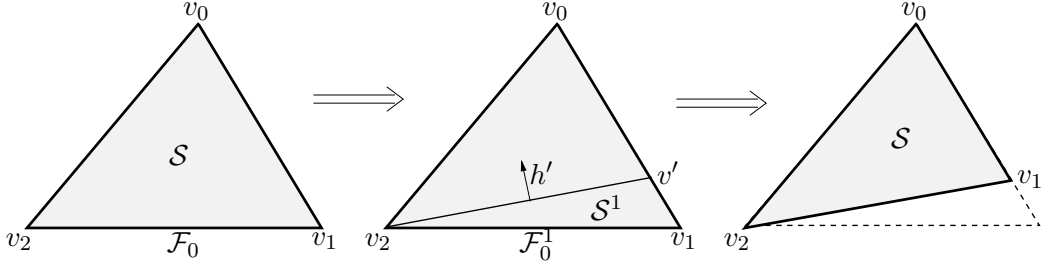}
\caption{Notation for the subdivision algorithm.}
\label{fig5}
\end{center}
\end{figure}

We have demonstrated the first step of a triangulation procedure that partitions $\cS$ into simplices on which sub-reach
control problems are solvable. Now we present a triangulation algorithm that iterates on the presented subdivision method.
It consists of $p$ iterations, one for each set $\{ v_{m_k},\ldots,v_{m_k+r_k-1} \}$, $k = 1,\ldots,p$. The notation 
$\cS^k := \conv \{ v',v_1,\ldots,v_n \}$ is understood to mean that all $n+1$ vertices of $\cS^k$ are 
assigned simultaneously in the order presented. The vertices of $\cS^k$ are later identified as
$\{ v_0^k,\ldots,v_n^k \}$. The algorithm generates simplices $\cS^1,\ldots,\cS^{p+1}$ starting from the
given simplex $\cS$. At the $k$th iteration, the current declaration of $\cS$ is split 
into a lower simplex $\cS^k$ and an upper simplex. The lower simplex is then ``thrown away'' and the remainder - 
the upper simplex - is declared to be $\cS$ with vertices called $\{ v_0,\ldots,v_n \}$ 
(overloading the vertices of the previous $\cS$). See Figure~\ref{fig5}. 
In this way each iterate mimics the first subdivision developed in the discussion above.

{\bf Subdivision Algorithm: } 
\label{triangle1} 
\begin{enumerate}
\item[1.]
Set $k = 1$.
\item[2.] 
Select $v' \in (v_0,v_{m_k})$ such that $\cB \cap \cone(\cS^k) \neq \Zero$, where
$\cS^k := \conv \{ v', v_1,\ldots,v_n \}$. 
\item[3.]
Set $\cS := \conv \{ v_0,v_1,\ldots,v_{m_k-1},v',v_{m_k+1},\ldots,v_n \}$.
\item[4.]
If $k < p$, set $k := k+1$ and go to step 2. 
\item[5.] 
Set $\cS^{p+1} := \cS$. 
\end{enumerate}

\begin{exmp}
Consider the output of the subdivision algorithm for an example with $p = 3$: 
\begin{itemize}
\item
$\cS^1 := \{ v_0^1,v_{m_1},\ldots,v_n \}$ where $v_0^1 \in (v_0,v_{m_1})$.
\item
$\cS^2 := \{ v_0^2, v_0^1, v_{m_1+1}, \ldots,v_n \}$ where $v_0^2 \in (v_0,v_{m_2})$.
\item
$\cS^3 := \{ v_0^3,v_0^1,v_{m_1+1},\ldots,v_0^2,v_{m_2+1},\ldots,v_n \}$ where $v_0^3 \in (v_0,v_{m_3})$.
\item
$\cS^4 := \{ v_0,v_0^1,v_{m_1+1},\ldots,v_0^2,v_{m_2+1},\ldots,v_0^3,v_{m_3+1},\ldots,v_n \}$. 
\end{itemize}
\end{exmp}

From this example we observe several features:
\begin{itemize}
\item
For each $k = 1,\ldots,p$, we have $v_0^k \in \cS^k \cap \cdots \cap \cS^{p+1}$ and
$v_{m_k} \in \cS^1 \cap \cdots \cap \cS^k$. 
\item
Simplex $\cS^{p+1}$  is the same as the originally given $\cS$ except that vertices 
$v_{m_1},\ldots, v_{m_p}$ have been replaced by new vertices $v_0^1,\ldots,v_0^p$, respectively.
\item
Because of the previous property $G^{p+1} := \cS^{p+1} \cap \cO$ has dropped in dimension to 
$\kappa-p = \mh - 1$ because $p$ vertices originally in $\cO$ have been removed from $\cS^{p+1}$. 
\end{itemize}

Let $\cF_0^k = \conv\{ v_1^k,\ldots,v_n^k \}$ denote the exit facet of $\cS^k = \conv\{ v_0^k,\ldots,v_n^k \}$. 
The triangulation generated by the algorithm has the property that 
$\cS^k \cap \cS^{k-1} = \cF^k_0$, $k = 2,\ldots,p+1$, 
and closed-loop trajectories follow paths through simplices with decreasing indices.
Thus, $\cS \overset{\cS}{\longrightarrow} \cF_0$ is achieved by implementing affine
controllers that achieve $\cS^k \overset{\cS^k}{\longrightarrow} \cF_0^k$ for $k = 1,\ldots,p+1$. 
In order to guarantee that switching occurs in the proper sequence (with decreasing simplex indices),
and to avoid chattering caused by measurement errors, a {\em discrete supervisor} should 
accompany the implementation of the piecewise affine feedback. The supervisor enforces the following
rule:
\begin{enumerate}
\item[(DS)]
At a point $x \in \cS$ belonging to more than one simplex $\cS^j$, the controller
for the simplex with the higher index is used.
\end{enumerate}

\begin{thm}
\label{thm:pwa}
Suppose Assumption~\ref{assum2} and \eqref{eq:r1}-\eqref{eq:BG1} hold.
If the invariance conditions for $\cS$ are solvable, then 
$\cS \overset{\cS}{\longrightarrow} \cF_0$ by piecewise affine feedback.
\end{thm}

\begin{proof}
Form the triangulation $\{ \cS^1,\ldots, \cS^{p+1} \}$ of $\cS$ based on the Subdivision Algorithm. 
%To show that $\cS \overset{\cS}{\longrightarrow} \cF_0$ by piecewise affine feedback, we first show that
%$\cS^k \overset{\cS^k}{\longrightarrow} \cF_0^k$ by affine feedback for $k = 1,\ldots,p+1$. 
We show by induction that $\cS^k \overset{\cS^k}{\longrightarrow} \cF_0^k$ 
by affine feedback for $k = 1,\ldots,p$ (momentarily ignoring the rule (DS)). 
For the initial step, by assumption the invariance conditions for $\cS$
are solvable and by Lemma~\ref{lem10}, $h_0 \cdot b_{m_k} > 0$ for $k = 1,\ldots,p$. Thus, by 
Lemma~\ref{lem3}, $\cS^1 \overset{\cS^1}{\longrightarrow} \cF_0$ by affine feedback. 
Now assume that at the $j$th step the invariance conditions are solvable for 
(the current) $\cS$ and $h_0 \cdot b_{m_k} > 0$ for $k = 1,\ldots,p$. Then
by Lemma~\ref{lem3}, $\cS^{j} \overset{\cS^j}{\longrightarrow} \cF_0^j$ by affine feedback. 
Now consider the $(j+1)$th step. By the algorithm $\cS := \conv \{ v_0,v_1,\ldots,v_{m_j-1},v',v_{m_j+1},\ldots,v_n \}$ 
and $h_0 = -h'$, where $v'$ and $h'$ are provided by the $j$th step. 
By Lemma~\ref{lem5}, the invariance conditions
are solvable for $\cS$ and $h_0 \cdot b_{m_k} > 0$ for $k = 1,\ldots,p$. Then by 
Lemma~\ref{lem3}, $\cS^{j+1} \overset{\cS^{j+1}}{\longrightarrow} \cF_0^{j+1}$ by affine feedback. 

Next consider $\cS^{p+1}$. We observe that $\cS^{p+1}$ and $\cS$ share the same invariance 
conditions since they only differ in their exit facets, so the invariance conditions for $\cS^{p+1}$
are solvable. Now let $\cG^{p+1} := \cS^{p+1} \cap \cO$. Then by the algorithm,
$\cG^{p+1} = \conv\{ v_2,\ldots,v_{m_2-1},v_{m_2+1},\ldots,v_{m_p-1},v_{m_p+1},\ldots,v_{\kappa+1} \}$. 
We can see that the algorithm has removed the $p$ vertices $v_{m_1},v_{m_2},\ldots,v_{m_p}$ from $\cG$. 
There remain $\mh$ linearly independent vectors in $\cB$ associated with $\cG^{p+1}$ (an $(\mh-1)$-dimensional simplex) given by 
$\{ b_2,\ldots,b_{m_2-1},b_{m_2+1},\ldots,b_{m_p-1},b_{m_p+1},\ldots,b_{\kappa+1} \}$. 
Therefore, we can apply Theorem~6.7 of \cite{MEB10} to obtain
$\cS^{p+1} \overset{\cS^{p+1}}{\longrightarrow} \cF_0^{p+1}$. 

Next, we must prove that trajectories progress through simplices with decreasing indices only. 
Consider w.l.og. the boundary between $\cS^1$
and $\cS^2$ given by $\cF_0^2 = \conv\{ v', v_2,\ldots,v_n \}$, and let $u = K_1 x + g_1$ be the affine
feedback obtained for $\cS^1$. 
We show that for any $x_0 \in \cS^1 \setminus \cF_0^2$, closed-loop trajectories do not reach 
$\cF_0^2$. This in turn means that trajectories never return to $\cS^2$ from $\cS^1$ after leaving $\cS^2$. 
This can be deduced from the proof of Lemma~\ref{lem3} where it is shown that the controls
$\{ u', u_2,\ldots,u_n \}$ can be selected so that 
$h' \cdot (A v' + B u' + a) < 0$ and $h' \cdot (A v_i + B u_i + a) < 0$, $i = 2, \ldots, n$. 
By convexity, $h' \cdot (A x + B (K_1 x + g_1) + a) < 0$ for all $x \in \cF_0^2$, from which the 
result easily follows.

Finally we verify conditions (ii) and (iii) of RCP. Condition (ii) follows immediately 
because there is a finite number of affine feedbacks each defined 
on a compact set $\cS^k$ that does not contain an equilibrium. For (iii)
we must verify that the piecewise affine feedback $u = f(x)$ resulting from (DS) 
satisfies \eqref{eq:inv2}. We show that it satisfies \eqref{eq:inv}, and by convexity also \eqref{eq:inv2}. 
First consider $\cS^{p+1}$. Its exit facet is
\[
\cF_0^{p+1} = \{ v_0^1,v_{m_1+1},\ldots,v_{m_1+r_1-1},\ldots,
v_0^p,v_{m_p+1},\ldots,v_{m_p+r_p-1},v_{r+1},\ldots,v_n \}.
\] 
The invariance conditions for $\cS^{p+1}$ are identical to those for $\cS$ and the 
controller for $\cS^{p+1}$ takes precedence over controllers for simplices with lower index. 
This implies the invariance conditions for $\cS$ hold at $v_0$ and all vertices of $\cF_0^{p+1}$.
The only vertices of $\cF_0$ that are not in $\cF_0^{p+1}$ are $v_{m_1}, v_{m_2},\ldots, v_{m_p}$. For these vertices we
have: $v_{m_1} \in \cS^1$, $v_{m_2} \in \cS^1 \cap \cS^2$,...,$v_{m_p} \in \cS^1 \cap \cdots \cap \cS^p$. 
We use the affine controller for the simplex with the highest index. But the invariance conditions 
for $\cS^k$ at $v_{m_k}$ are precisely those for $\cS$. We can see this because the invariance 
conditions for $v_{m_k}$ do not include the normal vector $-h'$ 
given in Lemma~\ref{lem8}.
\end{proof}

%\begin{exmp}
%The fact that the piecewise affine feedback in Theorem~\ref{} satisfies the invariance conditions
%\eqref{eq:inv2} for $\cS$ is illustrated in Figure~\ref{fig6}. 
%\end{exmp}

The main result of the paper stated next is that piecewise affine feedbacks are a sufficiently rich class to solve 
RCP when it is solvable by open-loop controls. The proof shows by a process of elimination 
that either RCP is solvable by affine feedback \cite{HVS06,RB06,MEB10} or it is solvable by 
(discontinuous) PWA feedback via the Subdivision Algorithm.

\begin{thm}
Suppose Assumption~\ref{assum1} holds. Then the following are equivalent:
\begin{enumerate}
\item
$\cS \overset{\cS}{\longrightarrow} \cF_0$ by piecewise affine feedback.
\item
$\cS \overset{\cS}{\longrightarrow} \cF_0$ by open-loop controls.
\end{enumerate}
\end{thm}

\begin{proof}
(1) $\Longrightarrow$ (2) is obvious. \\
(2) $\Longrightarrow$ (1) \
Suppose $\cS \overset{\cS}{\longrightarrow} \cF_0$ by open-loop controls. By Theorem~\ref{thm:nec1}, 
the invariance conditions are solvable. Let $\cG := \cS \cap \cO$. If $\cG = \emptyset$, then by 
Theorem~6.1 of \cite{MEB10}, $\cS \overset{\cS}{\longrightarrow} \cF_0$ by affine feedback. Suppose instead 
$\cG \neq \emptyset$. If $\cB \cap \cone(\cS) \neq \Zero$, then by Theorem~6.2 of \cite{MEB10}, 
$\cS \overset{\cS}{\longrightarrow} \cF_0$ by affine feedback. Suppose instead $\cB \cap \cone(\cS) = \Zero$.
From Theorem~\ref{thm:nec2}, $v_0 \not\in \cG$, so by reordering indices, $\cG = \conv\{ v_1,\ldots,v_{\kappa+1} \}$, 
where $0 \le \kappa < n$. 
Let $\{ b_1,\ldots,b_{\mh} ~|~ b_i \in \cB \cap \cC(v_i) \}$ be a maximal linearly independent set as
in (A3).
If $\kappa < \mh$, then by Theorem~6.7 of \cite{MEB10}, $\cS \overset{\cS}{\longrightarrow} \cF_0$ 
by affine feedback. Suppose instead $\kappa \ge \mh$. 
By Theorem~\ref{thm:nec2}, $\cB \cap \cC(v_i) \neq \Zero$ for 
$i \in \{ 1,\ldots,\kappa+1 \}$. Then Assumption~\ref{assum2} holds 
and the reach control indices can be defined. 
By Theorem~\ref{thm:pwa}, $\cS \overset{\cS}{\longrightarrow} \cF_0$ 
by piecewise affine feedback. 
\end{proof}

\section{Examples}
\label{sec:examples}

\subsection{Example 1}
Consider the system 
\[
\dot{x} = 
\left[
\begin{array}{rr}
0 & 1  \\
0  & 0 \\
\end{array}
\right] x + 
\left[ \begin{array}{r} 0 \\ 1 \end{array} \right] u +
\left[ \begin{array}{r}  0 \\ 0 \end{array} \right] \,. 
\]
Safety constraints on both $x_1$ and $x_2$ determine a polyhedral state space 
within which the dynamics evolve. The polyhedral state space is triangulated according 
to Assumption~\ref{assum1}. We focus on the reach control problem for a specific simplex of the triangulation: 
consider the simplex $\cS$ determined by vertices $v_0 =(-1,1) $, $v_1 = (1,0)$ and $v_2 = (0,0)$. 
It can be verified that $\cO = \{ x \in \RR^2 ~|~ x_2  = 0 \}$, $\cG = \conv\{v_1,v_2 \}$, 
$\kappa = 1$, and $\mh = m = 1$. Also $\cB \cap \cone(\cS) = \Zero$.
By the results of \cite{MEB10}, RCP is not solvable by continuous state feedback. 
For example, suppose we choose control values $u_0 = -\frac{3}{4}$, $u_1 = -1$, and $u_2 = 1$
to satisfy the invariance conditions \eqref{eq:inv}. By the method in \cite{HVS04}, this yields an affine feedback
$u = \left[ \begin{array}{rr} -2 & -3.75 \end{array} \right] x + 1$. 
Simulation of the closed-loop system is shown in Figure~\ref{example4}(a). 
We observe there exists a closed-loop equilibrium point on $\cG$. Now we show the problem is 
solvable by piecewise affine feedback. 

Let $b_1 = (0,-1) \in \cB \cap \cC(v_1)$. Since $h_0 = (0,-1)$, we have $h_0 \cdot b_1 > 0$,
verifying Lemma~\ref{lem10}. Next, we choose $v' = (0.5,0.25)$ along the simplex edge $(v_0,v_1)$
such that from Lemma~\ref{lem8}, $h' = (-0.25, 0.5)$. Then $h' \cdot b_1 < 0$ and 
$b_1 \in \cB \cap \cone(\cS^1)$, verifying Lemma~\ref{lem9}. 
Let $\cS^1 := \conv\{ v',v_1,v_2 \}$, $\cS^2 := \conv\{ v_0, v',v_2 \}$, and 
$\cF'_0 = \conv \{v',v_2\}$.
To satisfy the invariance conditions for $\cS^1$ we choose control inputs at the vertices to be
$u' =  -1$, $u_1 = -1$, and $u_{12} = -1$.
Similarly, for $\cS^2$ we choose $u_0 =  -\frac{3}{4}$, $u' = -1$, and $u_{22} = 1$. 
The piecewise affine feedback is 
\begin{equation*}
u :=  
\left\{
\begin{array}{ll}
\left[ \begin{array}{rr} 0 & 0 \end{array} \right] x - 1 \,,            & x \in \cS^1 \\
\left[ \begin{array}{rr} -2.0833 & -3.833 \end{array} \right] x + 1 \,, & x \in \cS^2 \,.
\end{array}
\right.
\end{equation*}
By Theorem~6.2 of \cite{MEB10}, $\cS^1 \overset{\cS^1}{\longrightarrow} \cF_0$ using $u$. 
Because $\cG^2 := \cS^2 \cap \cO = \{ v_2 \}$, we have $\mh^2 = 1$ and $\kappa^2 = 0$ for $\cS^2$. 
By Theorem~6.2 of \cite{MEB10}, $\cS^2 \overset{\cS^2}{\longrightarrow} \cF_0'$ using $u$. 
The closed-loop vector field is shown in Figure~\ref{example4}(b), where it is clear that RCP is solved.

\begin{figure}[!t] %  figure placement: here, top, bottom, or page
\begin{center}
\subfigure[]{\includegraphics[width = 0.48\linewidth]{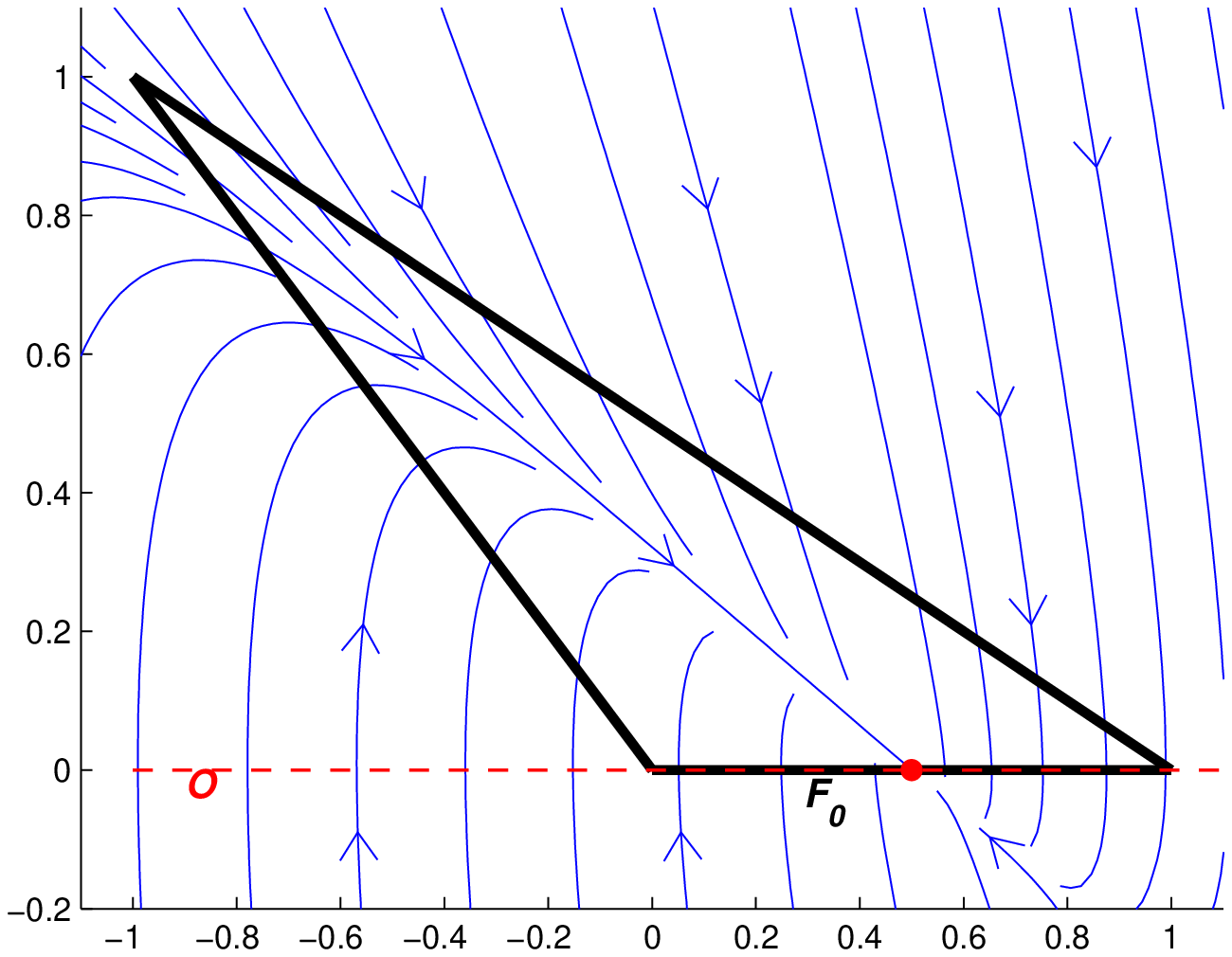}}
%\hskip 0.05\textwidth
\subfigure[]{\includegraphics[width = 0.48\linewidth]{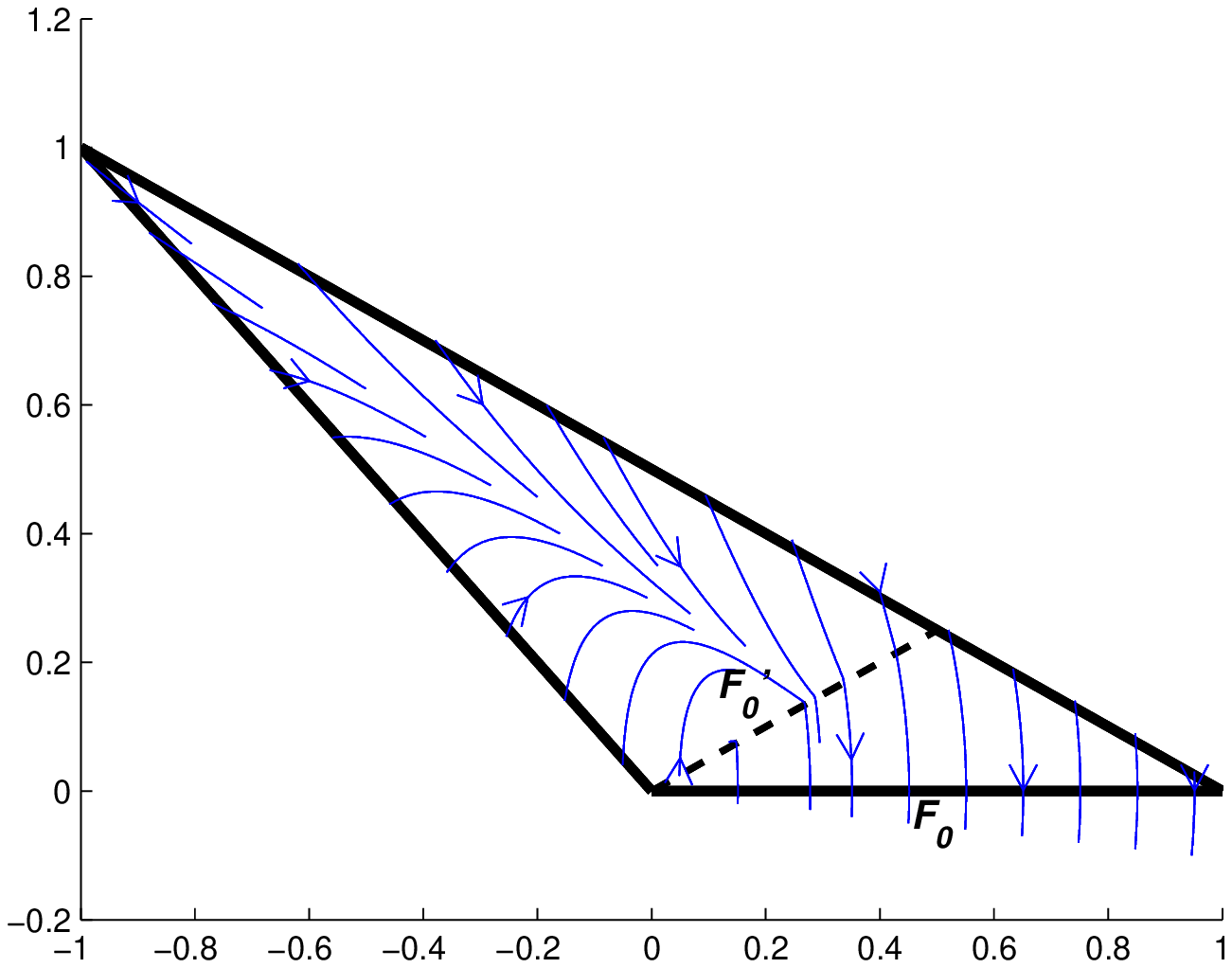}}
\caption{Closed-loop vector fields using (a) affine feedback and (b) piecewise affine feedback.}
\label{example4}
\end{center}
\end{figure}

\subsection{Example 2}
Consider the simplex $\cS$ in $\mathbb{R}^4$ defined by the vertices $v_0 =(0,0,0,0)$, 
$v_1 = (1,0,0,0)$,  $v_2 = (0,1,0,0)$, $v_3 =(0,0,1,0)$, and $v_4 =(0,0,0,1)$. 
%The outward normal vectors of the facets of $\cS$ are 
%$h_0 = (1,1,1,1)$, $h_1 = (-1,0,0,0)$, $h_2 = (0,-1,0,0)$, $h_3 = (0,0,-1,0)$, and $h_4 = (0,0,0,-1)$.  
Consider the system 
\[
\dot{x} = 
\left[
\begin{array}{rrrr}
-3 & -3 & -3 & 1 \\
0  & 0 & 0  & -2\\
-3 & -3 & -3 & 1 \\
0  & 0 & 0  & -2\\
\end{array}
\right] x + 
\left[ \begin{array}{rr} 0 & -2 \\ 0 & 1 \\ -2 & 0 \\ 1 & 0 \end{array} \right] u +
\left[ \begin{array}{r}  1 \\ 1 \\ 1 \\ 1 \end{array} \right] \,. 
\]
We compute $\cO = \{ x \in \RR^4 ~|~ x_1 + x_2 +x_3 + x_4-1 =0 \}$. 
Thus, $\cG = \cF_0$, and we note that $\kappa = 3$, $\mh = m = 2$, and $\cB \cap \cone(\cS) = \Zero$. 
By the results of \cite{MEB10}, RCP is not solvable by continuous state feedback. Now we show 
it is solvable by piecewise affine feedback.
First we examine the structure of $\cB$ (note that indices are not reordered, as is the convention
in our proofs). We find by inspect that $b_1 := (-2,1,0,0) \in \cB \cap \cC(v_1)$,
$b_3 := (0,0,-2,1) \in \cB \cap \cC(v_3)$, and $\cB = \spn \{ b_1, b_3 \}$. 
In particular, $b_2 := -b_1 \in \cB \cap \cC(v_2)$ and
$b_4 := -b_3 \in \cB \cap \cC(v_4)$. Thus, $r_1 = 2$ and $r_2 = 2$.

\subsubsection{First subdivision}

In the first iteration $\cS$ is subdivided into simplices $\cS^1$ and $\cS'$. 
Since $b_2 \cdot h_0 > 0$, we choose $v' = (0,0.75,0,0) \in (v_0,v_2)$ 
such that we obtain the condition $\cB \cap \cone(\cS^1) \neq \Zero$.  Hence 
$\cS'$ = conv$\{v_0,v_1,v',v_3,v_4\}$ and $\cS^1$ = conv$\{v',v_1,v_2,v_3,v_4\}$. 
%\begin{eqnarray*}
%h_1 = (-1,0,0,0),   \quad h' = (-1, -1.33,-1,-1), \quad h_3 = (0,0,-1,0), \quad h_4 = (0,0,0,-1)  
%\end{eqnarray*}
In order to satisfy the invariance conditions for $\cS^1$ the control inputs at the vertices of $\cS^1$ 
are chosen as $u' = (-1,-2)$, $u_{11} = (-1,-2)$, $u_{12}= (-1,-2)$, $u_{13}=(-1,-2)$, and 
$u_{14}=(1,0)$. This yields an affine feedback
\[
u := 
\left[ \begin{array}{rrrr} 0 & 0 & 0 & 2 \\ 0 & 0 & 0 & 2  \end{array} \right] x +\left[ \begin{array}{r} -1 \\ 
       -2   \end{array} \right]  \,, \qquad \qquad x \in \cS^1 \,.
\] 
For $\cS^1$ the invariance conditions are solvable and $\cB \cap \cone(\cS^1) \neq \Zero$, so by 
Theorem~6.2 of \cite{MEB10}, $\cS^1 \overset{\cS^1}{\longrightarrow} \cF_0$ using $u$. 
For $\cS'$ we have $\cG' := \cS' \cap \cO = \conv \{v_1,v_3,v_4\}$. 
Since $\kappa' = 2$ and $m = 2$, RCP is not solvable by continuous state feedback on $\cS'$,
and further subdivision of $\cS'$ is required.

\subsubsection{Second subdivision} 
 
Consider the simplex $\cS' = \conv\{ v_0, v_1, v' , v_3, v_4 \}$, where 
$v' \in (v_0,v_2)= (0,0.75,0,0)$ and the exit facet is $\cF_0'$ =conv$\{v_1,v', v_3,v_4\}$.   
We subdivide $\cS'$ into simplices $\cS^3$ and $\cS^2$ and use a piecewise affine feedback law to 
solve RCP on $\cS'$.  It is clear that $b_4 \cdot h_0' > 0$ and therefore we can choose 
$v'' \in (v_0,v_4)$ such that $\cB \cap \cone(\cS^2) \neq \Zero$. 
One choice is $v'' := (0,0,0,0.8)$. 
Let $\cS^3 = \conv\{v_0,v_1,v',v_3,v''\}$ and $\cS^2 = \conv\{v'',v_1,v',v_3,v_4\}$. 
It can be verified that $b_4 \in \cB \cap \cone(\cS^2)$. 
To satisfy the invariance conditions for $\cS^2$ we choose $u'' = (-4,0.6)$, $u_{21} = (-5,-1)$, 
$u'=(-1,-2)$, $u_{23} = (-5,-1)$, and $u_{24}=(-3,1)$. 
To satisfy the invariance conditions for $\cS^3$ we choose $u_0 = (0,0)$, $u_{31} = (-1,0)$, 
$ u' = (-1,-2)$, $u_{33} = (0,-1)$, and $u'' = (-4,0.6)$. 
This yields a piecewise affine feedback 
\[
u =
\left\{
\begin{array}{ll} 
\left[ \begin{array}{rrrr} -1 & -1.33 & 0 & -5 \\ 0 & -2.66 & -1 & 0.75 \end{array} \right] x  \,,
& x \in \cS^3 \\
\left[ \begin{array}{rrrr} 3 & 9.33 & 3 & 5 \\ 0 & -1.33 & 0 & 2  \end{array} \right] x +
\left[ \begin{array}{r} -8  \\ -1   \end{array} \right]  \,,
& x \in \cS^2 \,.
\end{array}
\right.
\] 
For $\cS^2$ the invariance conditions are solvable and $\cB \cap \cone(\cS^2) \neq \Zero$, so by 
Theorem~6.2 of \cite{MEB10}, $\cS^2 \overset{\cS^2}{\longrightarrow} \cF_0'$ using $u$. 
For $\cS^3$ we have $\cG_3 = \cS^3 \cap \cO = \conv \{v_1,v_3\}$. Since $\kappa^3 = 1$ and $\mh^3 = 2$,
by Theorem~6.7 of \cite{MEB10}, $\cS^3 \overset{\cS^3}{\longrightarrow} \cF''$ using $u$. 
Indeed, $\{ b_1, b_3 ~|~ b_i \in \cB \cap \cC(v_i) \}$ is a linearly 
independent set associated with $\cG_3$.

\section{Conclusion}
The paper studies the reach control problem on simplices, and we investigate cases when the problem is not
solvable by continuous state feedback. It is shown that the class of piecewise affine feedbacks is sufficient
to solve the problem in all cases of interest; namely, those cases when the problem is solvable by open-loop 
controls.

\bibliographystyle{IEEEtranS}

\section{Appendix}
\begin{proof}[Proof of Theorem~\ref{thm:nec1}]
Let $x_0 \in \cS \setminus \cF_0$. By assumption there exists $\mu_{x_0}(t)$ 
and a time $T(x_0) > 0$ such that $\phi_{\mu_{x_0}}(t,x_0) \in \cS$ for all $t \in [0,T(x_0)]$.
Since $\mu_{x_0}(t)$ is an open-loop control, there exists $c \ge 0$ such that 
$\| \mu_{x_0}(t) \| \le c$, for all $t \in [0,T(x_0)]$. 
Define $\cY(x) := \bigl\{ A x + B w + a ~|~ w \in \RR^m \bigr\}$ 
and $\cY_c(x) := \bigl\{ A x + B w + a ~|~ w \in \RR^m, \| w \| \le c \bigr\}$. 
Now take a sequence $\{ t_i ~|~ t_i \in (0,T(x_0)] \}$ with $t_i \rightarrow 0$. Since 
$\{ y \in \cY_c(x) ~|~ x \in \cS \}$ is bounded, there exists $M > 0$ such that
$\| \phi_{\mu_{x_0}}(t_i,x_0) - x_0 \| \le M t_i$. 
Therefore $\{ \frac{\phi_{\mu_{x_0}}(t_i,x_0) - x_0}{t_i} \}$
is a bounded sequence, and there exists a convergence subsequence (with indices relabeled) such that 
$\lim_{i \rightarrow \infty} \frac{\phi_{\mu_{x_0}}(t_i,x_0) - x_0}{t_i} =: v$. 
Since $\phi_{\mu_{x_0}}(t_i,x_0) \in \cS$, by the definition of the Bouligand tangent cone, 
$v \in T_{\cS}(x_0)$. On the other hand, we have
\begin{equation}
\label{eq:limit1}
\frac{\phi_{\mu_{x_0}}(t_i,x_0) - x_0}{t_i} = \frac{1}{t_i} \int_0^{t_i} \bigl[A \phi_{\mu_{x_0}}(\tau,x_0) + 
B {\mu_{x_0}}(\tau) + a \bigr] d \tau \,.
\end{equation}
Taking the limit, we get
\begin{equation*}
\label{eq:limit2}
v = A x_0 + B \lim_{i \rightarrow \infty} {\mu_{x_0}}(t_i) + a \in \cY(x_0) \,.
\end{equation*}
We conclude that $\cY(x_0) \cap T_{\cS}(x_0) \neq \emptyset$, $x_0 \in \cS \setminus \cF_0$. 
Since $T_{\cS}(v_0) = \cone(\cS)$, and $T_{\cS}(x) = \cC(v_i)$ for $x \in (v_0,v_i)$, 
it follows that the invariance conditions are solvable at $v_0$ and along simplex edges $(v_0,v_i), i \in I$. 

Now consider $v_i, i \in I$. 
%It is easy to show that $x \mapsto \cY_c(x)$ is upper semicontinuous.
If $v_i \in \cO$, then the invariance conditions are solvable by selecting $u_i \in \RR^m$ such 
that $A v_i + B u_i + a = 0$. Instead suppose $v_i \not\in \cO$. Suppose by way of contradiction that
$\cY(v_i) \cap \cC(v_i) = \emptyset$. Then $\cY(v_i)$ and $\cC(v_i)$ are non-empty disjoint polyhedral convex sets in
$\RR^n$. By Corollary 19.3.3 of \cite{ROCK}, they are strongly separated. That is, there exists $\epsilon > 0$ such that 
$\inf_{y \in \cY(v_i), z \in \cC(v_i)} \| y - z \| > \epsilon$. 
By the upper semicontinuity of $x \mapsto \cY(x)$, there exists $\delta > 0$ such that if 
$\| x - v_i \| < \delta$, then $\cY(x) \subset \cY(v_i) + \frac{\epsilon}{2} \Bs$. 
In particular, taking $x \in (v_0,v_i)$, we get 
$\cY(x) \cap \cC(v_i) = \emptyset$, a contradiction. 
\end{proof}

\begin{proof}[Proof of Theorem~\ref{thm:nec2}]
Consider $v_i \in V \cap \cG$. Suppose $\cB \cap \cC(v_i) = \Zero$. 
Since $A v_i + a \in \cB$, there exists $u_i \in \RR^m$ such that $Av_i + Bu_i + a = 0$. 
By (ii)-(iii) of Definition~\ref{def:ol}, there exists $\varepsilon > 0$ such that for 
all $x \in \cS \setminus \cF_0$, there exists $u_x \in \RR^m$ such that 
$A x + B u_x + a \in T_{\cS}(x)$ and $\| A x + B u_x + a \| > \varepsilon$.
By continuity there exists $\delta > 0$ such that if $\| x - v_i \| < \delta$, 
then $\| A x + Bu_i + a \| < \varepsilon/2$. 
Thus, for $x \in \cS \setminus \cF_0$ with $\| x - v_i \| < \delta$, 
we have $\| B (u_x - u_i) \| > \varepsilon/2$. 
Since $\cB \cap \cC(v_i) = \Zero$ and $\cC(v_i)$ is a closed cone, 
there exists $\alpha > 0$ such that if 
$b \in \cB$ satisfies $\| b \| > \varepsilon/2$, 
then $(b + \alpha \Bs) \cap \cC(v_i) = \emptyset$. In particular, we can choose $x \in (v_0,v_i)$ 
sufficiently close to $v_i$ such that $\| A x + Bu_i + a \| < \min\{ \alpha, \epsilon/2 \}$. Then 
$A x + B u_x + a = (A x + B u_i + a) + B(u_x - u_i) \not\in \cC(v_i) = T_{\cS}(x)$, 
a contradiction. 
\end{proof}

\end{document}